\documentclass[12pt]{amsart}

\usepackage{latexsym, amssymb, amsmath}

\usepackage{amsfonts, graphicx}
\usepackage{graphicx,color}
\newcommand{\be}{\begin{equation}}
\newcommand{\ee}{\end{equation}}
\newcommand{\beq}{\begin{eqnarray}}
\newcommand{\eeq}{\end{eqnarray}}

\newtheorem{thm}{Theorem}[section]

\newtheorem{lma}{Lemma}[section]

\newtheorem{cor}{Corollary}[section]
\newtheorem{defn}{Definition}[section]
\theoremstyle{remark}
\newtheorem{rem}{Remark}[section]
\numberwithin{equation}{section}

\def\dps{\displaystyle}

\def\be{\begin{equation}}
\def\ee{\end{equation}}
\def\bee{\begin{equation*}}
\def\eee{\end{equation*}}
\def\ol{\overline}
\def\lf{\left}
\def\ri{\right}

\def\K{K\"ahler }
\def\KR{K\"ahler-Ricci }
\def\Ric{\text{\rm Ric}}
\def\Rm{\text{\rm Rm}}

\def\wt{\widetilde}

\def\p{\partial}
\def\ddbar{\partial\bar\partial}

\def\ol{\overline}
\def\heat{\lf(\frac{\p}{\p t}-\Delta\ri)}
\def\tr{\operatorname{tr}}

\def\e{\epsilon}
\def\a{{\alpha}}

\def\ijb{{i\bar{j}}}
\def\ii{\sqrt{-1}}
\def\R{\mathbb{R}}
\def\C{\mathbb{C}}
\begin{document}
\title[]
{Chern-Ricci flows on noncompact complex manifolds}

 \author{Man-Chun Lee}
\address[Man-Chun Lee]{Department of
 Mathematics, The Chinese University of Hong Kong, Shatin, Hong Kong, China.}
\email{mclee@math.cuhk.edu.hk}

\author{Luen-Fai Tam$^1$}
\address[Luen-Fai Tam]{The Institute of Mathematical Sciences and Department of
 Mathematics, The Chinese University of Hong Kong, Shatin, Hong Kong, China.}
 \email{lftam@math.cuhk.edu.hk}
\thanks{$^1$Research partially supported by Hong Kong RGC General Research Fund \#CUHK 14305114}

\renewcommand{\subjclassname}{
  \textup{2010} Mathematics Subject Classification}

\date{August, 2017}

\begin{abstract} In this work, we obtain existence criteria for Chern-Ricci flows on noncompact manifolds. We generalize a result by Tossati-Wienkove \cite{TosattiWeinkove2015} on Chern-Ricci flows to  noncompact manifolds and a result for \KR flows by Lott-Zhang \cite{LottZhang2011} to Chern-Ricci flows. Using the existence results, we prove that any complete noncollapsed \K metric with nonnegative bisectional curvature on a  noncompact complex manifold can be deformed to a complete \K metric with nonnegative and {\it bounded} bisectional curvature which will have maximal volume growth if the initial metric has maximal volume. Combining this result with \cite{ChauTam2006}, we give another proof that a complete noncompact \K manifold with nonnegative bisectional curvature (not necessarily bounded) and maximal volume growth is biholomorphic to $\C^n$. This last result has already been proved by Liu \cite{Liu2017} recently using other methods. This last result is partial confirmation of a uniformization conjecture of Yau \cite{Yau1991}.

\end{abstract}


\maketitle

\markboth{Man-Chun Lee and Luen-Fai Tam}{Chern-Ricci flows on noncompact manifolds}
\section{introduction}\label{s-intro}

Let $(M^n,g_0)$ be a Hermitian manifold with complex dimension $n$. Let $\theta_0$ be the   associated real $(1,1)$ form. In holomorphic coordinates $(z^1,\dots,z^n)$, $g_0$ is given by $(g_0)_\ijb$ and
$$
\theta_0=\ii(g_0)_\ijb dz^i\wedge d\bar z_j.
$$
Even though $g_0$ is not \K in general, we still call $\theta_0$ to be the \K form of $g_0$.

 The following flow, which was called the Chern-Ricci flow by Tosatti-Weinkove \cite{TosattiWeinkove2015}, was first studied by Gill \cite{Gill2011}: starting at $\theta_0$ is given by
\begin{equation}\label{e-Gill}
\left\{
  \begin{array}{ll}
  \dps{ \frac{\partial\theta}{\partial t}}=&-\Ric(\theta)   \\
     \theta(0)=&\theta_0,
  \end{array}
\right.
\end{equation}
where $\Ric(\theta)$ is the Chern-Ricci form of $\theta$.  For a Hermitian metric with \K form $\theta$,
 $\Ric(\theta)=-\sqrt{-1}\partial\bar\partial \log (\det g)$ in   holomorphic coordinates. We also denote $\Ric(\theta)$ to be $\Ric(g)$.  For the definitions and basic properties on Chern connection, torsion of the connection and curvature of the connection, see \cite{TosattiWeinkove2015} or Appendix \ref{chern connection appendix}.

If initially $g_0$ is K\"ahler, then the Chern-Ricci flow coincides with the Ricci flow in which the K\"ahlerity is preserved. In fact, if $g_0$ is \K on an open set $U$ of $M$, then $g(t)$, which is the Hermitian metric corresponding to $\theta(t)$ of the solution to the Chern-Ricci flow, is also \K on $U$.
Unlike Ricci flow, Chern-Ricci flow will   preserve the Hermitian condition. Thus it is expected that the Chern-Ricci flow reveals information about the structure of $M$ as a complex manifold.

Existence, longtime existence and behaviors on compact Hermitian manifolds have been investigated by Gill \cite{Gill2011},   Tosatti, Weinkove, Fang, Yang and other people \cite{FangTosattiWeinkoveZheng2016,TosattiWeinkove2013,TosattiWeinkove2015,TosattiWeinkoveYang2015}. In this paper, we will discuss existence of Chern-Ricci flow on noncompact complex manifolds.

 In this work, connection always means the Chern connection, and curvature is the curvature with respect to the Chern connection unless specified otherwise. Also we use $[0,S]$ for example instead of $[0,T]$ to denote a  time interval because we want to reserve $T$ to denote the torsion.

In \cite{TosattiWeinkove2015}, Tosatti-Weinkove obtained the following criteria of existence time.
\begin{thm}\label{t-TosattiWeinkove} Let $(M^n,g_0)$ be a compact Hermitian manifold. Let $S_{A}$ be the supremum of $S>0$ so that the Chern-Ricci flow \eqref{e-Gill} has a solution $g(t)$ with initial data $g_0$ on $M\times[0,S]$. Let $S_{B}$ be defined by $$S_B=\sup\{ t\geq 0:\, \exists u\in C^\infty(M),\; \theta_0-tRic(\theta_0)+\sqrt{-1}\partial\bar\partial u>0\}.$$
Then $S_A=S_B$.
\end{thm}
This generalizes a result for \KR flows on compact  manifolds by  Tian-Zhang \cite{TianZhang2006} to the Chern-Ricci flows. The result by Tian-Zhang was also generalized in another direction: Namely to \KR flow on complete noncompact \K manifolds. First note that if $M$ is compact then for $u$ be as in the definition of $S_B$, we actually have:
$$
\beta'\theta_0\ge\theta_0-tRic(\theta_0)+\sqrt{-1}\partial\bar\partial u\ge \beta\theta_0
$$
for some $\beta, \beta'>0$. Moreover, all the derivatives of $u$ with respect to $g_0$ are bounded. In \cite{LottZhang2011}, Lott-Zhang proved the following:

\begin{thm}\label{t-LottZhang}
 Let $(M^n,g_0)$ be a complete noncompact \K manifold with bounded curvature and with \K form $\theta_0$. Let $S_A$  be the supremum   of the numbers $S$ such that the \KR flow has a solution $g(t)$ with initial data $g_0$ and with uniformly bounded curvature on $M\times[0,S]$.
Let $S_B$  be the supremum   of the numbers $S$  for which there is a
function $u\in C^\infty(M)$ such that

(i)
$$
 \theta_0-tRic(\theta_0)+\sqrt{-1}\partial\bar\partial u\ge \beta\theta_0;
$$
for some $\beta>0$; and

(ii) for each $k$, the $k$-th covariant derivatives of $u$ with respect to $g_0$ are uniformly bounded.

Then
$$S_A=S_B.$$

\end{thm}
Observe that by (ii) in the above and the fact that $g_0$ has bounded curvature, we also have
$$
\theta_0-tRic(\theta_0)+\sqrt{-1}\partial\bar\partial u\le \beta'\theta_0
$$
for some $\beta'>0$.

In this work, we want to generalize Theorem \ref{t-TosattiWeinkove} by Tosatti-Weinkove to noncompact manifolds and Theorem \ref{t-LottZhang} by Lott-Zhang   to Chern-Ricci flow. Before we state our result, we make the following definition.
\begin{defn}\label{d-bisectional}
Let $(M^n,g_0)$ be a Hermitian manifold. Let $R$ be the curvature tensor with respect to the Chern connection. $g_0$ is said to have bisectional curvature bounded below by $-K$, and will be denoted by $\mathrm{BK}(g_0)\ge -K$ if at any point and for any unitary frame, $R_{i\bar ij\bar j}\ge -K$. Here $i$ may be equal to $j$. The bisectional curvature of $g_0$ is bounded above by $K$ is defined similarly.
\end{defn}

Let $(M^n,g_0)$ be a complete noncompact Hermitian manifold. Let $S_{A}$ be the supremum of $S>0$ so that the Chern-Ricci flow \eqref{e-Gill}  has a solution $g(t)$ with initial data $g_0$ and $g(t)$ is uniformly equivalent to $g_0$ on $M\times[0,S]$. Let $S_{B}$ be the supremum of $S>0$ such that there is a smooth bounded function $u$ and $\beta>0$ such that
\bee
\theta_0-S\Ric(\theta_0)+\ii\ddbar u\ge \beta\theta_0
\eee
where $\theta_0$ is the \K form of $g_0$. We obtain the following:

\begin{thm}\label{t-existence-intro-1}
Let $(M^n,g_0)$ be a complete noncompact Hermitian manifold with torsion $T$ (which is zero if $g_0$ is K\"ahler). Assume the followings:
\begin{enumerate}
  \item [(i)]  $|T|_{g_0}$ and $ |\bar\p T|_{g_0}$ are uniformly bounded;
  \item [(ii)] the bisectional curvature of $g_0$ is uniformly bounded from below: $\mathrm{BK}(g_0)\ge -K$; and
  \item [(iii)] there exists a smooth real function $\rho$ which is uniformly equivalent to the distance function from a fixed point such that $|\p\rho|_{g_0}, |\ddbar \rho|_{g_0}$ are uniformly bounded.
\end{enumerate}
 Then $S_{A}=S_{B}$.
\end{thm}

Here we do not assume that $g_0$ has bounded Riemannian curvature. We only assume that the bisectional curvature of $g_0$ is bounded from below.

It is well-known that if $g_0$ is \K and  has bounded curvature,  then (iii) is always true, see \cite{Tam2010} for example. If $g_0$ is \K with nonnegative bisectional curvature then (iii) is also true, see \cite{NiTam2013}.

For the case that the bisectional curvature is bounded from above, we have:

\begin{thm}\label{t-existence-intro-2}
Let $(M^n,g_0)$ be a complete noncompact Hermitian manifold and let $T$ be the torsion of $g_0$. Suppose $g_0$ satisfies conditions (i) and (iii) in  Theorem \ref{t-existence-intro-1} and (ii) is replaced by the condition that $\mathrm{BK}(g_0)\le K$. Suppose in addition there is a smooth bounded function $u$ and positive constants $\beta, \beta'>0$ so that
$$
\beta'\theta_0\ge \theta_0-S\Ric(\theta_0)+\ii\ddbar u\ge \beta\theta_0.
$$
Then the Chern-Ricci flow has a solution $g(t)$ with initial data $g_0$ on $M\times[0,S_1]$ for some $S_1>0$. Moreover, $g(t)$ is uniformly equivalent to $g_0$.
\end{thm}
 In case $g_0$ has bounded curvature with $|T|_{g_0}$ and $ |\bar\p T|_{g_0}$ uniformly bounded, then one can obtain more explicit estimates. See Theorem \ref{t-existence-2} for more details. The estimates will be crucial in the proof of our next result on the short time existence of \KR flow.

Recall that in \cite{Cabezas-RivasWilking2011}, without assuming that the initial metric has bounded curvature, Cabezas-Rivas and   Wilking are able to construct a short time solution to the Ricci flow on complete noncompact manifolds starting from a metric with nonnegative complex sectional curvature and noncollapsing in the sense that the volume of every  geodesic ball of radius 1 is bounded below by a fixed positive constant $v_0$. Moreover, the curvature of the solution $g(t), t>0$ is bounded by $a/t$ for some $a$. The result is recently generalized to the case that the complex sectional curvature is bounded from below by Bamler, Cabezas-Rivas and   Wilking \cite{BamlerCabezasWilking2017}. On the other hand, Simon and Topping \cite{SimonTopping2017} use some ideas of Hochard \cite{Hochard2016} and together with their results in \cite{SimonTopping2016}  to prove that  similar results are true for three dimensional Riemannian manifolds under the   weaker condition that the Ricci curvature is bounded from below.
In this work, we prove the following:
\begin{thm}\label{existenceKRF}
Suppose $(M^n,g_0)$ is a complete noncompact \K manifold with complex dimension n with $BK\geq 0$ and  $V_0(x,1)\geq v_0>0$ for some $v_0>0$ for all $x\in M$. Then the following are true:

 \begin{enumerate}
   \item [(i)] There exist $ S=S(n,v_0)>0 , a(n,v_0)>0$ depending only on $n, v_0$ such that the \K Ricci flow has a complete  solution $g(t)$  on $M\times[0,S]$ and satisfies
$$  |Rm|(x,t)\leq \frac{a(n,v_0)}{t}$$
on $M\times (0,S].$
   \item [(ii)]   $g(t)$ has nonnegative bisectional curvature. If $g_0$ has maximal volume growth, then $g(t)$ also has maximal volume growth.
       \item[(iii)] $V_t(x,1)\ge \frac 12v_0$ on $M\times(0,S]$, where $V_t(x,1)$ is the volume of ball of radius 1 with center at $x$ with respect to $g(t)$.
 \end{enumerate}
\end{thm}
Combining this with the  a result of Chau and the second author \cite{ChauTam2006}, we have:
\begin{cor}\label{c-Yau}
Suppose $(M^n,g_0)$ is a complete noncompact \K manifold with nonnegative bisectional curvature with maximal volume growth, then $M$ is biholomorphic to $\C^n$.
\end{cor}

This gives a partial answer to the uniformization conjecture by Yau \cite{Yau1991} which states that a complete noncompact K\"ahler manifold with positive holomorphic bisectional curvature is biholomorphic to $\mathbb{C}^n$. In case $g_0$ has bounded curvature, the above result was proved in \cite{ChauTam2006}. Without assuming the boundedness of curvature, the above result is already proved recently by Gang Liu \cite{Liu2017}, where he uses Gromov-Hausdorff convergence together with other techniques. Here we give an alternative proof using \KR flow. We should mention that there are many important contributions by various authors on Yau's conjecture. For more detailed discussion on the results related to the conjecture, one might refer to \cite{ChauTam2008}.

The paper is organized as follows: in Section \ref{s-shorttime}, we will give a short time existence on Chern-Ricci flow. In Section \ref{s-estimates}, we derive some a priori estimates for Chern Ricci flow and apply it in Section \ref{s-existence} to show Theorems \ref{t-existence-intro-1} and \ref{t-existence-intro-2}. In Section \ref{s-KRF}, we will prove Theorem \ref{existenceKRF}. In the appendix, we collect some information about the Chern connection and the behaviors under conformal change of metrics.

{\it Acknowledgement}: The authors would like to thank Albert Chau for some useful comments. The authors also thank  Huai-Dong Cao, Albert Chau and Shing-Tung Yau for their interest in the work.

\section{a short time existence lemma}\label{s-shorttime}

Let $(M^n,g_0)$ be a complete noncompact Hermitian manifold with complex dimension $n$. In the following, connection and  curvature etc. always mean the Chern connection and curvature with respect to the Chern connection.
In this section, we want to discuss the existence of the Chern-Ricci flow:
\be\label{e-CRflow-1}
\left\{
  \begin{array}{ll}
    \frac{\p}{\p t}g_\ijb=  &  -R_\ijb; \\
    g(0)=  &g_0.
  \end{array}
\right.
\ee
where $R_\ijb=-\p_i\p_{\bar j}\log \det (g(t))$ is the Chern Ricci curvature of $g(t)$.
This equation is equivalent to the following parabolic complex Monge-Amp\`ere equation:
\be\label{e-MA-1}
\left\{
  \begin{array}{ll}
    \frac{\p}{\p t}\psi=  &  \dps{\log \frac{(\theta_0-t\Ric(\theta_0)+\ii\p\bar\p \psi)^n}{\theta_0^n} }; \\
    \psi(0)=  &0.
  \end{array}
\right.
\ee
More precisely, if $g(t)$ is a solution to \eqref{e-CRflow-1}, let
\be\label{e-psi}
\psi(x,t)=\int_0^t\log \lf(\frac {\theta^n(x,s)}{\theta_0^n(x)}\ri)ds
\ee
where $\theta(t)$ and $\theta_0$ are the \K forms of $g(t)$, $g_0$ respectively. Then $\psi$ satisfies \eqref{e-MA-1}. One can see that
  $\theta(t)= \theta_0-t\Ric(\theta_0)+\ii\p\bar\p \psi.$ Conversely, if $\psi$ is a smooth solution to \eqref{e-MA-1} so that $\theta_0-t\Ric(\theta_0)+\ii\p\bar\p \psi>0$, then $\theta(t)$ defined by the above relation satisfies \eqref{e-CRflow-1}. We will say that $\psi$ is the solution of \eqref{e-MA-1} corresponding to the solution $g(t)$ of \eqref{e-CRflow-1}.

  To begin our discussion, we need the following:

  \begin{defn}\label{boundedgeom} Let $(M^n, g)$ be a complete Hermitian manifold. Let  $k\ge 1$ be an integer and $0<\alpha<1$. $g$ is said to have   bounded geometry of   order $k+\alpha$ if there are positive numbers $r, \kappa_1, \kappa_2$  such that at every $p\in M$ there is a neighborhood $U_p$ of $p$, and local biholomorphism $\xi_{p}$ from $D(r)$ onto $U_p$ with $\xi_p(0)=p$ satisfying  the following properties:
  \begin{itemize}
    \item [(i)] the pull back metric $\xi_p^*(g)$   satisfies:
    $$
    \kappa_1 g_e\le \xi_p^*(g)\le \kappa_2 g_e$$ where $g_e$ is the standard metric on $\C^n$;
    \item [(ii)] the components $g_{\ijb}$ of $\xi_p^*(g)$ in the natural coordinate of $D(r)\subset \C^n$  are uniformly bounded in the standard $C^{k+\alpha}$ norm in $D(r)$ independent of $p$.
  \end{itemize}
$g$ is said to have bounded geometry of infinity order if instead of (ii) we have for any $k$, the $k$-th derivatives of $g_\ijb$ in $D(r)$ are bounded by a constant independent of $p$. $g$ is said to have bounded curvature of infinite order on a compact set $\Omega$ if (i) and (ii) are true for all $k$ for all $p\in \Omega$.
 \end{defn}

\begin{lma}\label{l-shortime-1} Let $(M^n,g_0)$ be a complete noncompact Hermitian metric. Suppose $g_0$ has bounded geometry of infinite order, then \eqref{e-MA-1} has a smooth solution $\psi$ on $M\times[0,S]$ for some $S>0$ and there is a constant $C>0$ such that
$C^{-1}\theta_0\le \theta(t)\le C\theta_0$ where $\theta_0$ is the \K form of $g_0$ and
$\theta(t)= \theta_0-t\Ric(\theta_0)+\ii\p\bar\p \psi$. In particular, $\theta(t)$ is a solution to \eqref{e-CRflow-1}.
\end{lma}
\begin{proof} The proof is similar to the proof of \cite[Lemma 2.2]{ChauTam2011}.

\end{proof}

We want to consider the following conditions on a complete noncompact Hermitian manifold $(M^n,g)$:

\begin{itemize}
  \item [{\bf(a1)}] There exists a smooth real function $\rho$ which is uniformly equivalent to the distance function from a fixed point such that $|\p\rho|_{g_0}, |\ddbar \rho|_{g_0}$ are uniformly bounded.
  \item [{\bf(a2)}] There is a smooth bounded function $u$ such that
  \bee
  \theta_0-S\Ric(\theta_0)+\ii\ddbar u\ge \beta\theta_0
  \eee
 where  $S, \beta$ are positive constants.
 \item [{\bf(a2)}] There is a smooth bounded function $v$ such that
  \bee
  \theta_0-S'\Ric(\theta_0)+\ii\ddbar v\le \beta'\theta_0
  \eee
  where  $S', \beta'$ are positive constants
\end{itemize}

We have the following:

\begin{lma}\label{l-equivalent} Let $g_0, g_1$  be two complete Hermitian metrics on a noncompact complex manifolds $M^n$. Suppose $\a g_0\le g_1\le \a^{-1}g_0$ for some $\a>0$.

\begin{enumerate}
  \item [(i)] Suppose $g_0$ satisfies {\bf (a1)}, so does $g_1$ with the same $\rho$.
  \item [(ii)] Suppose $g_0$ satisfies {\bf(a2)} with $S>0$, $u$ being bounded and $\beta>0$. Then $g_1$ also satisfies {\bf(a2)}:
       $$
       \theta_1-\a S\Ric(\theta_1)+\ii\ddbar u_1\ge \beta \theta_1
       $$
       for some smooth bounded function $u_1$.   Here $\theta_1$ is the \K form of $g_1$.
        \item [(iii)] Suppose $g_0$ satisfies {\bf(a3)} with $S'>0$, $v$ being bounded and $\beta'>0$. Then $g_1$ also satisfies {\bf(a2)}:
       $$
       \theta_1-\a S'\Ric(\theta_1)+\ii\ddbar v_1\ge \beta' \theta_1
       $$
       for some smooth bounded function $v_1$.
\end{enumerate}

\end{lma}
\begin{proof} (i) is obvious. To prove (ii), let $f=S\log(\theta_0^n/\theta_1^n)$, then for any $b>0$
\bee
\begin{split}
b\theta_1-S\Ric(\theta_1)+\ii\ddbar (u+f)=&b\theta_1-S\Ric(\theta_0)+\ii\ddbar u\\
\ge& \theta_1+(\beta-1)\theta_0 \\
\ge & (\a b+\beta-1)\theta_0\\
\ge & \a\beta g_1
\end{split}
\eee
if $b=\a^{-1}$. From this the result follows.

The proof of (iii) is similar.

\end{proof}

\section{a priori estimates for Chern-Ricci flow}\label{s-estimates}

Let $(M^n,h_0)$ be a Hermitian manifold and let $h(t)$ be a solution of \eqref{e-CRflow-1} with initial data $h(0)=h_0$ on $M\times[0,S]$ with $S>0$. Let $\omega_0$, $\omega(t)$ be the \K forms of $h_0, h(t)$ respectively. Let $\psi$ be the solution of \eqref{e-MA-1} corresponding to $h(t)$. Namely,
$$
\psi(x,t)=\int_0^t\log \lf(\frac {\omega^n(x,s)}{\omega_0^n(x)}\ri)ds
$$
and $\omega(t)=\omega_0-t\Ric(\omega_0)+\ii\ddbar\psi.$ We want to obtain some a priori estimates. First we list evolution equations of some quantities related to the Chern-Ricci flow.

\begin{lma}\label{l-psi}
Let $\Psi=t\dot\psi-\psi-nt$ and $\Lambda=(S_1-t)\dot\psi+\psi+nt$ where $\dot\psi=\frac{\p}{\p t}\psi$ and $S_1$ is any number. Then
\bee
\heat\Psi=-\tr_{h}h_0,
\eee
and
\bee
\heat\Lambda=-S_1\tr_h(\Ric(\omega_0))+\tr_h h_0
\eee
where $\Delta u=h^\ijb u_\ijb$.
\end{lma}
\begin{proof} Direct computations, see \cite{TosattiWeinkove2015} for example.
\end{proof}

Next let $\hat h$ be another fixed Hermitian metric.  We use the following notations:

\begin{itemize}
  \item $\hat \nabla$ is the derivative with respect to the Chern connection of $\hat h$.
  \item $\hat T, T_0$  are the torsions of $\hat h$, $h_0$ respectively
  \item $\hat R$ is the curvature of $\hat h$.
\end{itemize}
We want to compute $\heat \tr_{\hat h}h$ which has been obtained in \cite{TosattiWeinkove2015}. For later application, we also want to compute   $\heat \tr_{h}\hat h$.

\begin{lma}\label{l-trace} Let $\Upsilon=\tr_{\hat h}h$, and $\Theta=\tr_h\hat h$.

\begin{enumerate}
  \item [(i)]
  \bee
  \heat \log \Upsilon=\mathrm{I+II+III}
  \eee
  where
  \bee
\begin{split}
\mathrm{I}\le &2\Upsilon^{-2}\text{\bf Re}\lf(\hat h^{i\bar l}(h_0)_{p\bar l}h^{k\bar q}( T_0)_{ki}^p\hat \nabla_{\bar q}\Upsilon\ri).
\end{split}
\eee
\bee
\begin{split}
\mathrm{II}=&\Upsilon^{-1} h^\ijb \hat h^{k\bar l}h_{k\bar q} \lf(\hat \nabla_i \ol{(\hat T)_{jl}^p}- \hat h^{p\bar q}\hat R_{i\bar lp\bar j}\ri)\\
\end{split}
\eee
and
  \bee
\begin{split}
\mathrm{III}=&-\Upsilon^{-1} h^{\ijb}\hat h^{k\bar l}\lf(\hat \nabla_i\lf(\ol{( T_0)_{jl}^p}(h_0)_{k\bar p}\ri) +\hat \nabla_{\bar l}\lf( {(  T_0)_{ik}^p}(h_0)_{p\bar j}\ri)-\ol{ (\hat T)_{jl}^q}(  T_0)_{ik}^p(h_0)_{p\bar q} \ri)
\end{split}
\eee

  \item[(ii)]
  Suppose $\mathrm{BK}(\hat h)\le K$. Hence there exist positive constants  $C(n), C'(n)$ depending only on $n$, such that

\bee
\begin{split}
\heat\Theta\le& C(n)\Theta^2\lf(|\bar \p \hat T|_{\hat h}+K\ri)\\
\\&+C'(n)\Theta^3\lf(|\bar\p T_0|_{\hat h}|h_0|_{\hat h}+|T_0|_{\hat h}|\hat \nabla h_0|_{\hat h}+\Lambda^2\Upsilon|T_0|^2_{\hat h} + \Upsilon|\hat T|_{\hat h}^2\ri).
\end{split}
\eee
for some positive constants $C(n), C'(n)$ depending only on $n$, where $\Lambda=\tr_{h}h_0$
\bee
|\bar \p \hat T|_{\hat h}^2=\sum_{i,j,k,l}|\hat\nabla_{\bar l}\hat T_{ij}^k|^2
\eee
in a unitary frame of $\hat h$, and $|\bar\p T_0|_{\hat h}$ is defined similarly.
\end{enumerate}

\end{lma}
\begin{proof} (i) is by Tosatti-Weinkove \cite[Proposition 3.1]{TosattiWeinkove2015}.

(ii)

Since   $\hat \nabla \hat h=0$,

\bee
\begin{split}
\heat \Theta=& \hat h_{k\bar l} h^{k\bar q}h^{p\bar l}\lf(- \p_p\p_{\bar q}\log\det h+h^{\ijb} \hat\nabla_{i}\hat \nabla_{\bar j}h_{p\bar q}\ri)  \\
&+h^{\ijb}\hat h_{k\bar l}h^{k\bar q} (\hat \nabla_i h^{p\bar l})(\hat \nabla_{\bar j}h_{p\bar q})+h^{\ijb}\hat h_{k\bar l} h^{p\bar l} (\hat \nabla_ih^{k\bar q})(\hat \nabla_{\bar j}h_{p\bar q})
\end{split}
\eee

By \cite[p.135]{TosattiWeinkove2015},
\bee
\begin{split}
\hat\nabla_{i}\hat \nabla_{\bar j}h_{p\bar q}=&\hat\nabla_{\bar q}\hat\nabla_ph_\ijb-\hat R_{i\bar qp\bar s}\hat h^{m\bar s}h_{m\bar j}+\hat R_{i\bar qs\bar j}\hat h^{s\bar m s}h_{p\bar m}\\
&+\hat \nabla_{\bar q}\lf((T_0)_{ip}^m (g_0)_{m\bar j}\ri)-(\hat\nabla_{\bar q}\hat T_{ip}^m)h_{m\bar j}
-\hat T_{ip}^m \hat\nabla_{\bar q}g_{m\bar j}\\
&+\hat\nabla_i(\ol{(T_0)_{jq}^s}(h_0)_{p\bar s})-(\hat\nabla_i\ol{\hat T_{jq}^s})h_{p\bar s}-\ol{\hat T_{jq}^s}\hat\nabla_ih_{p\bar s}.
\end{split}
\eee

\bee
\begin{split}
\p_p\p_{\bar q}\log\det h=&h^\ijb\lf[ \hat\nabla_{\bar q}\hat\nabla_p h_\ijb-\hat R_{p\bar qi\bar s}\hat h^{m\bar s}h_{m\bar j}\ri]-h^{m\bar j}h^{i\bar s} \hat\nabla_p h_\ijb\hat\nabla_{\bar q}h_{m\bar s}.
\end{split}
\eee

Hence
\be\label{e-theta-1}
\begin{split}
&\heat \Theta\\
=&\hat h_{k\bar l}h^{k\bar q}h^{p\bar l}h^{m\bar j}h^{i\bar s} \hat\nabla_p h_\ijb \hat\nabla_{\bar q}h_{m\bar s} +\hat h_{k\bar l}h^{\ijb} h^{p\bar l}h^{k\bar q}\bigg[\hat \nabla_{\bar q}\lf((T_0)_{ip}^m (h_0)_{m\bar j}\ri)+\hat\nabla_i(\ol{(T_0)_{jq}^s}(h_0)_{p\bar s})
 \\
&-(\hat\nabla_{\bar q}\hat T_{ip}^m)h_{m\bar j}
-\hat T_{ip}^m \hat\nabla_{\bar q}h_{m\bar j}-(\hat\nabla_i\ol{\hat T_{jq}^s})h_{p\bar s}-\ol{\hat T_{jq}^s}\hat\nabla_ih_{p\bar s}\bigg]\\
&+\hat h_{k\bar l} h^{k\bar q}h^{p\bar l} \hat\nabla_{\bar q}\hat T_{pi}^i +\hat h_{k\bar l}h^\ijb h^{k\bar q}\hat h^{s\bar l}\hat R_{i\bar qs\bar j}\\
&+h^{\ijb}\hat h_{k\bar l}h^{k\bar q} (\hat \nabla_i h^{p\bar l})(\hat \nabla_{\bar j}h_{p\bar q})+h^{\ijb}\hat h_{k\bar l} h^{p\bar l} (\hat \nabla_ih^{k\bar q})(\hat \nabla_{\bar j}h_{p\bar q}).
\end{split}
\ee
where we have used the fact that

$$
\hat R_{p\bar qi\bar s}-\hat R_{i\bar qp\bar s}=\hat\nabla_{\bar q}\hat T_{pi\bar s}=\hat h_{\bar sr}\hat\nabla_{\bar q}\hat T_{pi}^r.
$$
Choose a unitary basis so that $h_{\ijb}=\delta_{ij}, \hat h_{ij}=\lambda_{i}\delta_{ij}$, then
\bee
\hat h_{k\bar l}h^{k\bar q}h^{p\bar l}h^{m\bar j}h^{i\bar s} \hat\nabla_p h_\ijb \hat\nabla_{\bar q}h_{m\bar s}=\sum_{i,j,k}\lambda_k\nabla_k h_\ijb \hat\nabla_{\bar k}h_{j\bar i}.
\eee
On the other hand,
\bee
\begin{split}
&h^{\ijb}\hat h_{k\bar l} h^{p\bar l} (\hat \nabla_ih^{k\bar q})(\hat \nabla_{\bar j}h_{p\bar q})+h^{\ijb}\hat h_{k\bar l}h^{k\bar q} (\hat \nabla_i h^{p\bar l})(\hat \nabla_{\bar j}h_{p\bar q})\\
= &-2\lambda_k\hat \nabla_ih_{k\bar j}\hat \nabla_{\bar i}h_{j\bar k}\\
=&-2\lambda_k\lf(\hat\nabla_kh_{i\bar j}+\lf((h_0)_{s\bar j}(T_0)_{ik}^s-h_{s\bar j}\hat T_{ik}^s\ri) \ri)\lf(\hat\nabla_{\bar k}h_{j\bar i}+\ol{\lf((h_0)_{s\bar j}(T_0)_{ik}^s-h_{s\bar j}\hat T_{ik}^s\ri) } \ri)\\
\le &-\frac32\sum_{i,j,k} \lambda_k|\hat\nabla_k h_\ijb|^2+c_1(n)\Theta\lf(\Lambda^2|T_0|^2_{ h}+ |\hat T|^2_h\ri).
\end{split}
\eee
for some constant $c_1(n)$ depending only on $n$, where we have used the fact that
\bee
\begin{split}
\hat\nabla_ih_{k\bar j}-\hat\nabla_kh_{i\bar j}= (h_0)_{s\bar j}(T_0)_{ik}^s-h_{s\bar j}\hat T_{ik}^s .
\end{split}
\eee

\bee
\begin{split}
-{ \hat h_{k\bar l}h^{\ijb} h^{p\bar l}h^{k\bar q}
(\hat T_{ip}^m \hat\nabla_{\bar q}h_{m\bar j}+\ol{\hat T_{jq}^s}\hat\nabla_ih_{p\bar s}) }
=&
-\sum_{i,j,k}\lambda_k(\hat T_{ik}^j \hat\nabla_{\bar k}h_{j\bar i}+\ol{\hat T_{ik}^j}\hat\nabla_ih_{k\bar j}) \\
\le &\frac12\sum_{i,j,k}\lambda_k|\hat\nabla_kh_{i\bar j}|^2  +c_2(n)\Theta\lf(\Lambda^2|T_0|^2_h+|\hat T|_h^2\ri)
\end{split}
\eee
for some $c_2(n)$ depending only on $n$.

So we have
\be\label{e-theta-2}
\begin{split}
&\heat \Theta\\
\le & \hat h_{k\bar l}h^{\ijb} h^{p\bar l}h^{k\bar q}\bigg[\hat \nabla_{\bar q}\lf((T_0)_{ip}^m (h_0)_{m\bar j}\ri)+\hat\nabla_i(\ol{(T_0)_{jq}^s}(h_0)_{p\bar s})
 -(\hat\nabla_{\bar q}\hat T_{ip}^m)h_{m\bar j}
 -(\hat\nabla_i\ol{\hat T_{jq}^s})h_{p\bar s} \bigg]\\
&+\hat h_{k\bar l} h^{k\bar q}h^{p\bar l} \hat\nabla_{\bar q}\hat T_{pi}^i +\hat h_{k\bar l}h^\ijb h^{k\bar q}\hat h^{s\bar l}\hat R_{i\bar qs\bar j}  +c_4(n)\Theta\lf(\Lambda^2|T_0|^2_{ h}+ |\hat T|^2_h\ri).
\end{split}
\ee
for some $c_3>0$ depending only on $n$.
Now choose unitary frame so that $\hat h_\ijb=\delta_{ij}$, $h_\ijb=\sigma_i\delta_{ij}$, then we can
\bee
\begin{split}
\hat h_{k\bar l}h^\ijb h^{k\bar q}\hat h^{s\bar l}\hat R_{i\bar qs\bar j}=&\sigma_i^{-1}\sigma_k^{-1}\hat R_{i\bar kk\bar i}\\
=&\sigma_i^{-1}\sigma_k^{-1}\lf(\hat R_{k\bar ki\bar i}-\hat\nabla_{\bar k}\hat T_{ki\bar i}\ri)\\
=&\sigma_i^{-1}\sigma_k^{-1}\lf(\hat R_{k\bar ki\bar i}-\hat\nabla_{\bar k}(\hat h_{p\bar i}\hat  T_{ki}^p)\ri).
\end{split}
\eee
Also,
\bee
|T_0|^2_{ h}\le  \Theta^2 \Upsilon |T_0|_{\hat h}^2, \ |\hat T|^2_{ h}\le  \Theta^2 \Upsilon |\hat T|_{\hat h}^2.
\eee
Hence if $\mathrm{BK}(\hat h)\le   K$, then
\be\label{e-theta-3}
\begin{split}
\heat\Theta\le& C(n)\Theta^2\lf(|\bar \p \hat T|_{\hat h}+K\ri)\\
\\&+C'(n)\Theta^3\lf(|\bar\p T_0|_{\hat h}|h_0|_{\hat h}+|T_0|_{\hat h}|\hat \nabla h_0|_{\hat h}+\Lambda^2\Upsilon|T_0|^2_{\hat h} + \Upsilon|\hat T|_{\hat h}^2\ri).
\end{split}
\ee
From this (ii) follows.

\end{proof}

We have the following maximum principle.
\begin{lma}\label{l-max} Let $(M^n,h_0)$ be a complete noncompact Hermitian manifold satisfying condition
{\bf (a1)}: There exists a smooth real function $\rho$ which is uniformly equivalent to the distance function from a fixed point such that $|\p\rho|_{h_0}, |\ddbar \rho|_{h_0}$ are uniformly bounded. Suppose $h(t)$ is a solution to the Chern-Ricci flow with initial metric $h(0)=h_0$ on $M\times[0,S)$. Assume for any $0<S_1<S$, there is $C>0$ such that
$$
C^{-1}h(t)\le h_0\le Ch(t)
$$
for $0<t\le S_1$.
Let $f$ be  a smooth function    on $M\times[0,S)$ which is bounded from above such that
$$
\heat f\le0
$$
on $\{f>0\}$. Suppose  $f\le 0$ at $t=0$, then $f\le 0$ on $M\times[0,S)$.
\end{lma}
\begin{proof} The proof is standard and we may assume that $h(t)$ is a solution up to $S$ and is uniformly equivalent to $h_0$ on $M\times [0,S]$.  Let $r(\cdot)$ be the distance function from a fixed point $x_0$. We may assume that there is $C_1>0$ such that
$$
C_1^{-1}(1+r(x))\le \rho(x)\le C_1(1+r(x)).
$$
and $\rho\ge 1$.
Since $h(t)$ is uniformly equivalent to $h_0$, there is $C_2>0$ such that $\Delta \rho\le C_2$.
Hence
\bee
\heat (e^{2C_2t}\rho)=e^{2C_2t}\lf(2C_2\rho-C_2\ri)\ge C_2e^{2C_2t} \rho
\eee
For any $\e>0$, if the $\sup_{M\times[0,T]}(f-\e  e^{2C_2t}\rho)>0$, then there is $(x_0,t_0)$ such with $t_0>0$ such that $f-\e  e^{2C_2t}\rho\le 0$ on $M\times[0,t_0]$ and $f-\e  e^{2C_2t}\rho=0$ at $(x_0,t_0)$. In particular, $f(x_0,t_0)>0$. Hence at $(x_0,t_0)$ we have
$$
0\le \heat(f-\e  e^{2C_2t}\rho)<0,
$$
which is impossible. Since $\e$ is arbitrary, we conclude that the lemma is true.
\end{proof}

Assume that $h_0$ satisfies {\bf (a1), (a2)} (with $g_0$ replaced by $h_0$). Let $S$ be the constant and let $u$ be the bounded function in {\bf (a2)}. Suppose $h(t)$ is a solution to the
Chern-Ricci flow on $M\times[0,S_1]$ with $S_1<S$ with initial data $h_0$ so that $h(t)$ is uniformly equivalent to $h_0$ on $M\times[0,S_1]$.    Let $\psi$ be the corresponding solution to \eqref{e-MA-1}. Then  $\dot\psi$ and hence $ \psi$ are uniformly bounded on $M\times[0,S_1]$.

\begin{lma}\label{l-psi-est} Suppose the curvature of $h_0$ with $\mathrm{BK}(h_0)\ge -K$ for any unitary frame at any point in $M$. Then there is a constant $c_1(n)>0$ depending only on $n$ such that for $t\le S_1$,
 \begin{enumerate}
   \item [(i)] $\psi \le \lf(\log (1+c_1(n)KS_1)^n+1\ri)t$
   \item [(ii)]  $\dot\psi\le \lf(\log (1+c_1(n)KS_1)^n+1\ri)+n$.
   \item [(iii)] $$\dot\psi(x,t)\ge \frac1{S-S_1} \lf[\inf_Mu-\sup_Mu- \lf(\log (1+c_1(n)KS_1)^n+1+n\ri)t\ri], $$
       and
       $$
       \psi(x,t)\ge \frac1{S-S_1}\int_0^t\lf[\inf_Mu-\sup_Mu- \lf(\log (1+c_1(n)KT)^n+1+n\ri)s\ri]ds.
       $$
 \end{enumerate}

\end{lma}
\begin{proof} We   modify the proof in \cite{TosattiWeinkove2015}. For $\e>0$, let $\Phi=\psi-At-\e \rho$ where $A>0$ to be determined. Then $\Phi<0$ on $(M\setminus \Omega\times[0,S_1])  $ for some compact set $\Omega$.  Hence if $\sup_{M\times[0,S_1]}\Phi>0$, then there is $(x_0,t_0)\in \Omega\times[0,S_1]$ with $t_0>0$ such that
$\Phi(x_0,t_0)=\sup_{M\times[0,S_1]}\Phi$. At $(x_0,t_0)$, we have $\ii(\ddbar \psi-\e\ddbar\rho)\le 0$, and we have
\bee
\begin{split}
0\le&\dps{\frac{\p}{\p t}\Phi } =  \dps{\log \frac{(\omega_0-t_0\Ric(\omega_0)+\ii\p\bar\p \psi)^n}{\omega_0^n} }-A.
\end{split}
\eee
 Now
 \bee
 \begin{split}
 \omega_0-t_0\Ric(\omega_0)+\ii\p\bar\p \psi\le &\omega_0-t_0\Ric(\omega_0)+ \ii\e\ddbar\rho \\
 \le&(1+c(n)Kt_0+\e C_1) \omega_0
 \end{split}
 \eee
for some positive constant $c(n)$ depending only on $n$ and a positive constant  $C_1$ so that $\ii\ddbar\rho\le C_1\omega_0$. Here we have used the fact that $\mathrm{BK}(h_0)\ge -K$ and $|\ddbar\rho|_{h_0}$ is bounded. Hence if $A=\log (1+c(n)KS_1)^n+1$, we have a contradiction when $\e$ is small enough. Therefore $\Phi\le 0$ is $\e$ is small enough. Let $\e\to0$, we conclude that (i) is true.

(ii) By Lemmas \ref{l-psi} and \ref{l-max}, we conclude that
$t\dot\psi-\psi-nt\le0$ because $\dot\psi, \psi$ are bounded and $\psi=0$ at $t=0$. From this and (i), (ii) follows.

(iii) Let $\Lambda=(S-t)\dot\psi+\psi+nt$. By Lemma \ref{l-psi},
\bee
\begin{split}
\heat (\Lambda-u)=&-S\tr_h(\Ric(\omega_0))+\tr_hh_0+\Delta u\\
\ge &-\tr_h(\omega_0+\ii\ddbar u)+\tr_hh_0+\Delta u\\
= &0.
\end{split}
\eee
Here we have used {\bf (a2)}. Since $\Lambda, u$ are bounded, we conclude by Lemma \ref{l-max} that
$$
\Lambda-u\ge  \inf_{M, t=0}(\Lambda-u)=-\sup_M u.
$$
From this and (i), we conclude that (iii) is true.
\end{proof}

\begin{lma}\label{l-trace-est-1} Let $h_0$ be as in Lemma \ref{l-psi-est}. Suppose in addition $|T_0|_{h_0}^2, |\bar\p T_0|_{h_0}$ are uniformly bounded by $K_1$. Then there are constants  $c_1(n), c_2(n)$ depending only on $n$ such that for $t\le S_1<S_2<S$, solution $h(t)$ of \eqref{e-CRflow-1} on   $M\times[0,S_1]$  which is uniformly equivalent to $h_0$, then on $M\times[0,S_1]$
\bee
\tr_{h_0}h\le \exp\lf[A+\log\lf(\frac12(c_1+(c_1^2+c_2K_1^2A)^\frac12)\ri)\ri]
\eee
where
$$
A=\a^{-1}(2\mathfrak{m}+1)^2(c_1(K+K_1)+1)
$$
and
$\alpha=1-\dps{\frac{S_2}{S}}$ and $\mathfrak{m}=\sup_{M\times [0,S_1]}|(S_2-t)\dot\psi+\psi+nt-\frac{S_2}Su|$.

\end{lma}
\begin{proof} We will modify the proof in \cite{TosattiWeinkove2015} to the noncompact case. Let $\Upsilon=\tr_{h_0}h$ and $\Theta=\tr_hh_0$. We first estimate $\heat \log\Upsilon$. In the following small case $c_i$ will denote a positive constant depending only on $n$. By Lemma \ref{l-trace}(i) with $h_0=\hat h$, in the notations as in the lemma, we have
\bee
\mathrm{II}\le \frac12c_3 \Theta(K+K_1).
\eee
\bee
\mathrm{III}\le \frac12c_3\Upsilon^{-1}K_1\Theta.
\eee
Hence
\be\label{e-trace-est-1}
\heat \log\Upsilon\le 2\Upsilon^{-2}\text{\bf Re}\lf( h^{k\bar q}( T_0)_{ki}^i\hat \nabla_{\bar q}\Upsilon\ri)+c_3(K+K_1+\Upsilon^{-1})\Theta
\ee
Fix $S_2<S$ and let $\Lambda=(S_2-t)\dot\psi+\psi+nt$ and let $u$ be as in {\bf(a2)}, then by Lemma \ref{l-psi},
\bee
\begin{split}
\heat (\Lambda-\frac{S_2}Su)=&-S_2\tr_{h}(\Ric(\omega_0))+\frac{S_2}S\Delta u+\Theta\\
\ge& -\frac {S_2}S\tr_{h}(\omega_0+\ii\ddbar u)+\frac{S_2}S\Delta u+\Theta\\
=&(1-\frac{S_2}S)\Theta.
\end{split}
\eee
Let $\mathfrak{m}=\sup_M|\Lambda-\frac{S_2}Su|$, and let $\phi=\Lambda-\frac{S_2}Su+\mathfrak{m}+1\ge 1$. By the proof of Lemma \ref{l-max}, there is a constant $C_1>0$ such that if $P=e^{2C_1t}\rho$, then
$$
\heat P\ge 0.
$$
Finally, let
$$
Q=\log\Upsilon+A\phi^{-1}-\e P
$$
where $A\ge 1$ is a constant to be determined later. Then
\be\label{e-trace-est-3}
\begin{split}
\heat Q\le &2\Upsilon^{-2}\text{\bf Re}\lf( h^{k\bar q}( T_0)_{ki}^i\hat \nabla_{\bar q}\Upsilon\ri)+c_3(K+K_1+\Upsilon^{-1})\Theta\\
&-A\phi^{-2}\heat \phi-2A\phi^{-3}h^\ijb \hat\nabla_i\phi\hat\nabla_{\bar j}\phi \\
\le &2\Upsilon^{-2}\text{\bf Re}\lf( h^{k\bar q}( T_0)_{ki}^i\hat \nabla_{\bar q}\Upsilon\ri)+c_3(K+K_1+\Upsilon^{-1})\Theta\\
&-A\a\phi^{-2}\Theta-2A\phi^{-3}h^\ijb \hat\nabla_i\phi\hat\nabla_{\bar j}\phi\\
\end{split}
\ee
where $\a=1-\frac{S_2}S>0$. Since $Q<0$ outside $\Omega\times[0,S_1]$ for some compact set and $Q$ is bounded from above, there is $(x_0,t_0)$ with $x_0\in \Omega$ such that $Q(x_0,t_0)=\sup_{M\times[0,S_1]}Q$. Suppose $t_0=0$, then
\be\label{e-trace-est-2}
\sup_{M\times[0,S_1]}Q\le A
\ee
because $\phi\ge 1$. Suppose $t_0>0$, then at $(x_0,t_0)$
$$
\Upsilon \hat\nabla_{\bar q}\Upsilon=A\phi^{-2}\hat\nabla_{\bar q}\phi+\e\hat\nabla_{\bar q}P.
$$
Hence in a unitary basis with respect to $h_0$ so that $(h_0)_\ijb$, $h_\ijb=\lambda_i\delta_{ij}$,
\bee
\begin{split}
2\Upsilon^{-2}\text{\bf Re}\lf( h^{k\bar q}( T_0)_{ki}^i\hat \nabla_{\bar q}\Upsilon\ri)=&\Upsilon^{-1}\text{\bf Re}\lf( h^{k\bar q}( T_0)_{ki}^i(A\phi^{-2}\hat\nabla_{\bar q}\phi+\e\hat\nabla_{\bar q}P)\ri)\\
\le
&c_4K_1\Upsilon^{-1}\lf(A\sum_q\lambda_q^{-1}(|\hat\nabla_{\bar q}\phi|+\e |\hat\nabla_{\bar q}P|) \ri)\\
\le&2A\phi^{-3}\sum_q\lambda_q^{-1}|\hat\nabla_q\phi|^2 +c_5K_1^2A\Upsilon^{-2}\phi^3\sum_q\lambda_q^{-1}+\e C_2\Theta \\
=&2A\phi^{-3}h^\ijb\hat\nabla_i\phi\hat\nabla_{\bar j}\phi +c_5K_1^2A\Upsilon^{-2}\phi^3\Theta+\e C_2\Theta
\end{split}
\eee
  for some constant $C_2$ independent of $\e$ and some constants $c_4, c_5>0$ depending only on $n$. By \eqref{e-trace-est-3}, we have at $(x_0,t_0)$
  \bee
  \begin{split}
  0\le&\heat Q\\
  \le &c_3(K+K_1+\Upsilon^{-1})\Theta-A\a\phi^{-2}\Theta+c_5K_1^2A\Upsilon^{-2}\phi^3\Theta+\e C_2\Theta.
  \end{split}
  \eee
  Hence
  \bee
 0\le c_5K_1^2A\Upsilon^{-2}\phi^3+c_3\Upsilon^{-1}+(c_3(K+K_1)+\e C_2-A\a \phi^{-2})
  \eee
  Let $A=\a^{-1}(2\mathfrak{m}+1)^2(c_3(K+K_1)+\e C_2+1)$, then
   \bee
 0\le c_5K_1^2A\Upsilon^{-2}\phi^3+c_3\Upsilon^{-1}-1.
  \eee
  Since $\Upsilon^{-1}>0$, we have
  \bee
  \Upsilon^{-1}\ge \displaystyle\frac{-c_3+(c_3^2+4c_5K_1^2A )^\frac12}{2c_5K_1^2A}.
  \eee
  because $\phi\ge1$. Hence

  \bee
  \Upsilon\le\frac12\lf(c_3+(c_3^2+4c_5K_1^2A )^\frac12\ri).
  \eee
 Therefore on $M\times[0,S_1]$

  \bee
  Q\le \log \lf(\frac12\lf(c_3+(c_3^2+4c_5K_1^2A )^\frac12\ri)\ri)+A
  \eee
  As $\epsilon$ is arbitrary, the result follows.
\end{proof}

\section{existence criteria for Chern-Ricci flow}\label{s-existence}

$(M^n,g_0)$ be a complete noncompact Hermitian manifold. Let $S_{A}$ be the supremum of $S>0$ so that the Chern-Ricci flow \eqref{e-CRflow-1} has a solution $g(t)$ with initial data $g_0$ such that $g(t)$ is uniformly equivalent to $g_0$ in $M\times[0,S]$.

Let $S_{B}$ be the supremum of $S>0$ such that there is a smooth bounded function $u$ satisfying {\bf(a2)}, that is:
\be\label{e-a2}
\theta_0-S\Ric(\theta_0)+\ii\ddbar u\ge \beta\theta_0
\ee
for some $\beta>0$, where $\theta_0$ is the \K form of $g_0$.

We want to prove the following:
\begin{thm}\label{t-existence-1}
Let $(M^n,g_0)$ be a complete noncompact Hermitian manifold with torsion $T$. Assume the following such that
\begin{enumerate}
  \item [(i)]  $|T|_{g_0}$ and $ |\bar\p T|_{g_0}$ are uniformly bounded;
  \item [(ii)] The bisectional curvature of $g_0$ is uniformly bounded from below.
  \item [(iii)] There exists a smooth function $\rho$ satisfying {\bf(a1)}.
\end{enumerate}
 Then $S_{A}=S_{B}$.
\end{thm}
The proof of $S_A\le S_B$ is easy, see \cite{LottZhang2011,TosattiWeinkove2015} for example.
It remains to prove that $S_B\le S_A$. Hence let $S>0$ be such that \eqref{e-a2} holds for some bounded function $u$ and for some $\beta>0$. Since it is in general not true that $g_0$ has bounded geometry of all order, we proceed as in \cite{LeeTam2017}.

Let $\kappa\in (0,1)$, $f:[0,1)\to[0,\infty)$ be the function:
\be\label{e-exh-1}
 f(s)=\left\{
  \begin{array}{ll}
    0, & \hbox{$s\in[0,1-\kappa]$;} \\
    -\displaystyle{\log \lf[1-\lf(\frac{ s-1+\kappa}{\kappa}\ri)^2\ri]}, & \hbox{$s\in (1-\kappa,1)$.}
  \end{array}
\right.
\ee
Let   $\varphi\ge0$ be a smooth function on $\R$ such that $\varphi(s)=0$ if $s\le 1-\kappa+\kappa^2 $, $\varphi(s)=1$ for $s\ge 1-\kappa+2 \kappa^2 $
\be\label{e-exh-2}
 \varphi(s)=\left\{
  \begin{array}{ll}
    0, & \hbox{$s\in[0,1-\kappa+\kappa^2]$;} \\
    1, & \hbox{$s\in (1-\kappa+2\kappa^2,1)$.}
  \end{array}
\right.
\ee
such that $\displaystyle{\frac2{ \kappa^2}}\ge\varphi'\ge0$. Define
 $$\mathfrak{F}(s):=\int_0^s\varphi(\tau)f'(\tau)d\tau.$$
From \cite{LeeTam2017}, we have:

\begin{lma}\label{l-exhaustion-1} Suppose   $0<\kappa<\frac18$. Then the function $\mathfrak{F}\ge0$ defined above is smooth and satisfies the following:
\begin{enumerate}
  \item [(i)] $\mathfrak{F}(s)=0$ for $0\le s\le 1-\kappa+\kappa^2$.
  \item [(ii)] $\mathfrak{F}'\ge0$ and for any $k\ge 1$, $\exp( -k\mathfrak{F})\mathfrak{F}^{(k)}$ is uniformly  bounded.
  \item [(iii)]  For any $ 1-2\kappa <s<1$, there is $\tau>0$ with $0<s -\tau<s +\tau<1$ such that
 \bee
 1\le \exp(\mathfrak{F}(s+\tau)-\mathfrak{F}(s-\tau))\le (1+c_2\kappa);\ \ \tau\exp(\mathfrak{F}(s_0-\tau))\ge c_3\kappa^2
 \eee
  for some absolute constants  $c_2>0, c_3>0$.
\end{enumerate}

\end{lma}

For   any  $\rho_0>0$, let  $U_{\rho_0}$ be the component of $ \{x|\ \rho(x)<\rho_0\}$ which contains the fixed point mentioned in {\bf(a1)}. Hence $U_{\rho_0}$ will exhaust $M$ as $\rho_0\to\infty$.

For $\rho_0>>1$, let $F(x)=\mathfrak{F}(\rho(x)/\rho_0)$. Let $h_0=e^{2F}g_0$. Then $(U_{\rho_0},h_0)$ is a complete Hermitian metric, see \cite{Hochard2016}, and $h_0=g_0$ if on $\{\rho(x)<(1-\kappa+\kappa^2)\rho_0\}$.
\begin{lma}\label{l-conformal-1}
\begin{enumerate} Suppose the torsion $T$ of $g_0$ satisfies $|T|^2_{g_0}\le K, |\bar\p T|_{g_0}\le K$ and $\mathrm{BK}(g_0)\ge -K$. Then
  \item [(i)] For any $\e>0$, the torsion   $T_0$ of $h_0$ satisfies $|T_0|_{h_0}^2\le K+\e$, $|\bar\p T_0|_{h_0}\le K+\e$,    and $\mathrm{BK}(h_0)\ge -K-\e$, provided $\rho_0$ is large enough.

  \item [(ii)] For any $\e>0$, there is $\rho_1>0$ such that if $\rho_0\ge \rho_1$, then
  $$
  \omega_0-(S-\e)\Ric(\omega_0)+\frac{S-\e}S\ii\ddbar u\ge \frac\beta2\omega_0
  $$
  where $\omega_0$ is the \K form of $h_0$.
\end{enumerate}

\end{lma}
\begin{proof}
By Lemma \ref{l-exhaustion-1} and   assumption (ii) in the Theorem.
\be\label{e-conformal-2}
\begin{split}
|\p F|_{h_0}=&e^{-F}|\p F|_{g_0}\\
=&e^{-F}\rho_0^{-1}\mathfrak{F}'|\p\rho|\\
\le &C_1\rho_0^{-1}
\end{split}
\ee
for some constant $C_1$ independent of $\rho_0$ by assumption (ii) in Theorem.

Similarly,
\be\label{e-conformal-3}
|\ddbar F|_{h_0}\le C_1\rho_0^{-1}
\ee
for a possible larger $C_1$. (i) and (ii) follow from Lemma \ref{l-conformal-a} and the fact that $F\ge0$.

 To prove (ii), we may assume that $\beta\le 1$, for any $\e>0$, then if $\rho_0$ large enough,
\bee
\begin{split}
\omega_0-(S-\e)\Ric(\omega_0)=&\omega_0-(S-\e)\Ric(\theta_0)+2n(S-\e)\ii\ddbar F\\
\ge &\omega+\frac{S-\e}{S}(\beta-1)\theta_0-\frac{S-\e}{S}\ii\ddbar u-\e e^{2F}\theta_0 \\
\ge&\left(1+\frac{S-\e}{S}(\beta-1)-\e\right)\omega_0-\frac{S-\e}{S}\ii\ddbar u
\end{split}
\eee
because $\omega_0\ge\theta_0$.

\end{proof}

We need the following technical lemma which will be proved later.
\begin{lma}\label{l-boundedgeom}
$(U_{\rho_0},h_0)$ has bounded geometry of infinite order.
\end{lma}

\begin{proof}[Proof of Theorem \ref{t-existence-1}] Let $S>0$ be such that \eqref{e-a2} is true. Let $0<\e<S$. Let $\rho_0$ be large enough so that Lemma \ref{l-conformal-1} is true. Let $S_\e=S-\e$.

 \underline{\bf Claim.} The Chern Ricci flow with initial data $h_0$ has solution $h(t)$ on $U_{\rho_0}\times[0,S_\e)$ provided $\rho_0$ is large enough. Moreover, for any $\delta>0$ such that $0<S_\e-\delta<S_\e$   there is a constant independent of $\rho_0$ and $\e$ such that
$$
C^{-1}h(t)\le h_0\le Ch(t)
$$
on $U_{\rho_0}\times[0,S_\e-\delta)$.

If the claim is true,  let $\e\to 0$ and choose  suitable $\rho_0(\e)\to\infty$. By the estimates \cite{ShermanWeinkove2013} we conclude that the solutions corresponding to $\rho_0(\e)\to\infty$  will subconverge uniformly on compact sets to a solution of the Chern-Ricci flow $g(t)$ with initial data $g_0$ on $M\times[0,S)$. Moreover, for any $0<S'<S$, $g(t)$ is uniformly equivalent to $g_0$ on $M\times[0,S']$. From this we see that $S_1\le S_2$ and Theorem \ref{t-existence-1} is true.

To prove the claim, since $h_0$ has bounded geometry of infinite order, the Chern-Ricci flow \eqref{e-CRflow-1} has a solution $h(t)$ on $U_{\rho_0}\times[0,S_1]$ for some $S_\e>S_1>0$ such that $h(t)$ is uniformly equivalent to $h_0$. Let $\delta>0$ such that $S_\e-\delta>0$ and $S_1\le S_\e-\delta$. Let $K$  be the bound as in Lemma \ref{l-conformal-1}. By Lemma \ref{l-psi-est} and \ref{l-trace-est-1}, there is a constant  $C_1$ depending only on $\delta, n, K, \mathfrak{m}, S$ such that
$$
\lf|\log\lf(\frac{\omega^n(t)}{\omega_0^n}\ri)\ri|\le C_1; \ \tr_{h_0}h\le C_1
$$
on $U_{\rho_0}\times [0,S_1]$ where $\omega(t), \omega_0$ are the \K forms of $h(t), h_0$ respectively and $\mathfrak{m}=\sup_M|u|$. Hence there is a constant $C_2$ depending only on  $\delta, n, K, \mathfrak{m}, S$ such that
\be\label{e-existence-1}
C_2^{-1}h_0\le h(t)\le C_2h_0
\ee
on $U_{\rho_0}\times[0,S_1]$. Since $h_0$ has bounded geometry of infinite order, by the estimates in \cite{ShermanWeinkove2013}, we conclude that all derivatives of $h(t)$ with respect to $h_0$ are uniformly bounded on $U_{\rho_0}\times[0,S_1]$. Moreover, $h(S_1)$ also has bounded geometry of infinite order. By Lemma \ref{l-shortime-1}, we conclude that $h(t)$ can be extended beyond $S_1$ to some $S_2$ with $S_1<S_2<S_\e$ so that $h(t)$ is uniformly equivalent to $h_0$ on $U_{\rho_0}\times[0,S_2]$. Combining with \eqref{e-existence-1}, we conclude that the claim is true. This completes the proof of the theorem.

\end{proof}

It remains to prove Lemma \ref{l-boundedgeom}: $h_0$ has bounded geometry of infinite order.

\begin{proof}[Proof of Lemma \ref{l-boundedgeom}] $g_0$ has bounded geometry of infinite order on any compact set $\Omega$ of $M$: That is, there exists $r>0,\kappa_1,\kappa_2$ so that (i), (ii) in Definition \ref{boundedgeom} are true for all $k$ and for all points in $\Omega$. In fact, if $x\in \Omega$, there is an open set $O_x$ and there is a biholomorphism   $\xi_x: D(2)\to O_x$ with $\xi_x(0)=x$ such that $\xi_x^*(h_0)$ satisfying (i),(ii)  in Definition \ref{boundedgeom}. Since $\Omega$ is compact, we may find finitely many $x_1,\dots, x_m$ such that $\cup_i\xi_{x_i}(D(1))\supset \Omega$. For any $x$, there is $i$, such that $\xi_{x_i}(x)\in D(1)$. From it is easy to see that $h_0$ has bounded geometry of infinite order on $\Omega$.

 Let $\rho_1>\rho_0$. $g_0$ has bounded geometry of infinite order on $\ol{U_{\rho_1}}$ with $r=1$ and for some $\kappa_1,\kappa_2$ as in the Definition \ref{boundedgeom}.

 Let $B(x,r)$ and $\hat B(x,r)$ be the geodesic balls of radius $r$ centered at $x$ with respect to $g_0$ and $h_0=e^{2F}g_0$ respectively. Recall $F(x)=\mathfrak{F}(\rho(x)/\rho_0)$. We will argue as in \cite{Hochard2016,He2016}. See \cite{LeeTam2017}. Let   $x\in U_{\rho_0}$.

 \underline{Case 1}: $\rho(x)\le (1-2\kappa)\rho_0 $. Then
 $\rho(y)\le (1-2\kappa)\rho_0+C_1d(x,y)$ for some constant $C_1$ independent of $x, y$. Here $d(x,y)$ is the distance function with respect to $g_0$. Hence there is a $r_1>0$ independent of $x$ such that if $\rho(x)\le (1-2\kappa)\rho_0$, then $\rho(y)<(1-\kappa+\kappa^2)\rho_0$ for all $y\in B(x,r_1)$. In particular, $B(x,r_1)\subset U_{\rho_0}$ and $F=0$ on $B(x,r_1)$. Since $g_0=h_0$ on $B(x,r_1)$, one can see that there is a $0<\sigma_1$ independent of $x$ such that $\xi_{x}:D(\sigma_1)\to U_{\rho_0}$ which is a biholomorphism and such that the $\xi_x^*(h_0)=\xi_x^*(g_0)$. Hence $h_0$ has bounded geometry of infinite order on $\{x\in U_{\rho_0}|\ \rho(x)\le (1-2\kappa)\rho_0\}$.

 \underline{Case 2}: Suppose $\rho_0>\rho(x)>(1-2\kappa)\rho_0$. Let $s=\frac{\rho(x)}{\rho_0}$ which satisfies $1-2\kappa<s<1$. Let $\tau$ be the number in Lemma \ref{l-exhaustion-1}(iii) for this given $s$. Then as in the proof of Theorem 2.2 in \cite{LeeTam2017}, there is $r_2>0$ independent of $x$ such that
 \be\label{e-conformal-1}
 \hat B(x,r_2)\subset\{y\in U_{\rho_0}|\ s-\tau<\frac{\rho(y)}{\rho_0}<s+\tau\}.
\ee
By Lemma \ref{l-exhaustion-1}(iii), we conclude that if $\a=e^{\mathfrak{F}(s-\tau)}$, then there is a constant $C_2>0$ independent of $x$ such that
$$
\a^2g_0\le h_0\le C^2_2\a^2g_0
$$
on $\hat B(x,r_2)$. Hence
\bee
B(x,C_2^{-1}\a^{-1}r_2)\subset \hat B(x,r_2)\subset B(x,\a^{-1}r_2).
\eee
Since $\a\ge 1$, we may assume that $r_2>0$ which is independent of $x$ such that $\xi_x(D(\a^{-1}\sigma_2))\supset B(x,\a^{-1}r_2)$ for some $1>\sigma_2>0$ independent of $x$. Moreover, there is $C_3$ which is independent of $x$ such that $$\xi_x(D(C_3\a^{-1}\sigma_2))\subset B(x,C_2^{-1}\a^{-1}r_2).
$$
Hence
$\xi_x$ is a biholomprhism from $D(C_3\a^{-1}\sigma_2)$ to an open neighborhood $O_x$ of $x$ which is a subset of $\hat B(x,r_2)$ so that $\xi_x(0)=x$.

Define $\zeta_x: D(\sigma_2)\to O_x$ by $\zeta_x(w)=\xi_x(C_3\a^{-1}w).$
Then $$\zeta_x^*(h_0)=C_3^2\a^{-2}\xi_x^*(h_0)=C_3^2\a^{-2}e^{2F}\xi_x^*(g_0).$$
Let $z=C_3\a^{-1}w$, $w\in D(\sigma_2)$. Let $(h_0)_\ijb$ be the components of $\zeta_x^*(h_0)$ of in the $w$ coordinates and let $(g_0)_\ijb$ be the components of $g_0$ of $\xi_x^*(g_0)$ in the $z$ coordinates, then
$$
(h_0)_\ijb(w)=C_3^2e^{2F}\a^{-2}(g_0)_\ijb(z).
$$
By Lemma \ref{l-conformal-1} and \eqref{e-conformal-1},  there is a constant $C_4>0$ independent of $x$ such that $\a\le e^{F(y)}\le C_4\a$ for all $y\in \hat B(x,r_2)$. Hence we conclude that $\kappa_1'g_e\le \zeta^*(h_0)\le \kappa_2'$ for some positive constants $\kappa_1', \kappa_2'$ independent of $x$ and $g_e$ is the standard Euclidean metric in the $w$-space.

On the other hand,
\bee
\begin{split}
\frac{\p}{\p w^k}(h_0)_\ijb(w)=&2C_3^2\rho_0^{-1} \mathfrak{F}'e^{2F}\a^{-3}\frac{\p\rho}{\p z^l} (g_0)_\ijb(z)+C_3^3e^{2F}\a^{-3} \frac{\p }{\p z^k}(g_0)_\ijb(z).
\end{split}
\eee
By Lemma \ref{l-conformal-1} and the fact that $\a\le e^{F(y)}\le C_4\a$ for all $y\in \hat B(x,r_2)$, we conclude that
  $|\frac{\p}{\p w^k}(h_0)_\ijb(w)|$ is uniformly bounded on $D(\sigma_2)$ by a constant independent of $x$. Similarly, one can prove that for all $k\ge1$, all $k$ derivatives of $(h_0)_\ijb$ are uniformly bounded on $D(\sigma_2)$ by a constant independent of $x$.

  Combining case 1 and case 2, we conclude that the lemma is true.

\end{proof}

In Theorem \ref{t-existence-1}, we assume that $\mathrm{BK}(g_0)$ is uniformly bounded below. In case $\mathrm{BK}(g_0)$ is uniformly bounded from above, then we still have short time existence provided $g_0$ satisfies {\bf(a1)--(a3)} in Section \ref{s-shorttime}.

\begin{thm}\label{t-existence-2}
Let $(M^n,g_0)$ be a complete noncompact Hermitian manifold and let $T$ be the torsion of $g_0$.

\begin{enumerate}
  \item [(i)] Suppose $g_0$ satisfies {\bf(a1)--(a3)} and suppose $\mathrm{BK}(g_0)$ is uniformly bounded from above, and  $|T|_{g_0}, |\bar\p T|_{g_0}$ are uniformly bounded. Then there is a solution $g(t)$ of \eqref{e-CRflow-1} on $M\times[0,S_1]$ for some $S_1>0$ which are uniformly equivalent to $g_0$.
  \item [(ii)] Suppose the curvature of $g_0$, $|T|_{g_0}$ and $|\bar\p T|_{g_0}$ are uniformly bounded by $K>0$. Suppose also the Riemannian curvature of $g_0$ is bounded.  Then there is a constant $\a(n)>0$ depending only on $n$ such that \eqref{e-CRflow-1} has a smooth solution $g(t)$ on $M\times[0,\a K^{-1}]$ such that $\a g_0\le g(t)\le \a^{-1} g_0$ on $M\times[0,\a K^{-1}]$.
\end{enumerate}
\end{thm}
\begin{proof}   It is easy to see that if $g_0$ satisfies the conditions in (ii) then it satisfies the conditions in (i). However, in (ii) we want to get more simple estimates for later use. Let us prove (ii) first. Since the Riemannian curvature of $g_0$,   and $|T|_{g_0} $ are uniformly bounded, we conclude that $g_0$ satisfies {\bf(a1)} by Lemma \ref{l-Chern-connection-2} and \cite{Tam2010}. Since the curvature is bounded from above by $K$, there is a constant $c_1(n)$ depending only on $n$ such that
$$
\theta_0-2c_1K^{-1}\Ric(\theta_0)\ge \beta\theta_0
$$
where $\theta_0$ is the \K form of $g_0$. Since the  curvature is also bounded from below by $-K$, and $|T|_{g_0}, |\bar\p T|_{g_0}$ are uniformly bounded, we can apply Theorem \ref{t-existence-1} to conclude that there is a solution $g(t)$ of the Chern-Ricci flow with initial data $g_0$ on $M\times[0,c_1 K^{-1}]$ so  that $g(t)$ is uniformly equivalent to $g_0$ on $M\times[0,c_1 K^{-1}]$. By Lemma \ref{l-psi-est}, with $S_1=c_1K^{-1}$ and $u=0$ there, we conclude that
\be\label{e-trace-est-4}
\lf|\frac{\theta^n(t)}{\theta_0^n}\ri|= |\exp(\dot\psi)|\le c_2
\ee
 on $M\times[0,c_1 K^{-1}]$ for some constant $c_2>0$ depending only on $n$, where $\psi$ is the solution of \eqref{e-MA-1} corresponding to $g(t)$ and $\theta(t)$ is the \K form of $g(t)$.

 Next we want to estimate $\Theta:=\tr_{g}g_0$. Let $\Upsilon=\tr_{g_0}g$. By \eqref{e-trace-est-4}, we have
 $$
 \Upsilon\le c(n)\Theta^{n-1}
 $$
 on $M\times[0,c_1 K^{-1}]$ for some constant $c(n)>0$ depending only on $n$. By Lemma \ref{l-trace} with $\hat g=g_0$, on  $M\times[0,c_1 K^{-1}]$, we have
\bee
\heat\Theta\le c_3 K \lf( \Theta^2+\Theta^3\lf(1+(1+\Theta^2)\Upsilon\ri)\ri)\le c_4K(\Theta+1)^{n+4},
\eee
for some positive constants $c_3, c_4$ depending only on $n$. Let $v(t)$ be a function of $t$ such that
$$
\frac{d}{dt}v=c_4Kv^{n+4}
$$
with $v(0)=n+1$. That is
$$
v(t)=\lf(\frac1{ (n+1)^{-n-3}- \lambda t }\ri)^{\frac 1{n+3}}
$$
where $\lambda=c_4(n+3)K$. $v(t)\ge0$ is defined for $0\le t\le \frac12\lambda^{-1}(n+1)^{-n-3}$ and is bounded by a constant depending only on $n$ on $[0,t_0]$ for all $t_0<\frac12\lambda^{-1}(n+1)^{-n-3}$. Fixed such $t_0$, since $\Theta$ is bounded on $M\times[0,c_1 K^{-1}]$,
$$
A=c_4K\sup_{M\times [0,t_1]}\sum_{i=0}^{n+4}(\Theta+1)^iv^{n+4-i}<\infty
$$
for $t_1=\min\{t_0,c_1K^{-1}\}$. Hence
\bee
\begin{split}
 \heat\lf( e^{-  At}(\Theta+1-v)\ri) =&e^{-\lambda At}\lf(\lambda((\Theta+1)^{n+4}-v^{n+4})- A(\Theta+1-v)\ri)\\
 =&e^{-\lambda At}\lambda(\Theta+1-v)\lf(\sum_{i=0}^{n+4}(\Theta+1)^iv^{n+4-i}-A\ri)\\
\le &0
\end{split}
\eee
at the point where $\Theta+1-v\ge0$. Since at $t=0$, $\Theta+1-v=0$, we have $\Theta-1+v\le 0$ by Lemma \ref{l-max}. Hence there is constant $c_5>0$ depending only on $n$  such that
$$
\Theta\le c_5
$$
on $M\times[0, \a K^{-1}]$, where
$$
\a=\min\{c_1, \frac14c_4^{-1}(n+1)^{-n-4}\}
$$
which depends only on $n$. Combining this with \eqref{e-trace-est-4}, we conclude that
$$
c_6 g_0\le g(t)\le c_6^{-1}g_0
$$
on $M\times[0,\a K^{-1}]$ for some positive constant $c_6$ depending only on $n$ by choosing $\a$ smaller than $c_6$.

The proof of (i) is similar to the proof of Theorem \ref{t-existence-1}. We only indicate the necessary modifications. Let $\rho$ be the function in {\bf (a1)} and for $\rho_0>>1$, let $F$, $h_0=e^{2F}g_0$ as in the proof of Theorem \ref{t-existence-1}. Using {\bf(a2), (a3)} as in the proof of Lemma \ref{l-conformal-1}, there exist $S_2>0$, $S_2'>0$  $\beta_1>0,  \beta'_1>0$ and smooth bounded functions $u_1, v_1$ which are independent of $\rho_0$ such that
\be\label{e-upper-1}
  \omega_0-S_2\Ric(\omega_0)+\ii\ddbar u_1\ge\beta\omega_0,
\ee
and
\be\label{e-upper-2}
 \omega_0-S_2'\Ric(\omega_0)+\ii\ddbar v_1\le \beta'\omega_0.
\ee
By Theorem \ref{t-existence-1}, there is a solution $h(t)$ to the Chern-Ricci flow which is uniformly equivalent to $h_0$ on $U_{\rho_0}\times[0, \frac{S_2}2]$. Similar to Lemma \ref{l-conformal-1}, $\mathrm{BK}(h_0)$, $|T_0|_{h_0}$ and $|\bar\p T|_{h_0}$ are uniformly bounded from above by a constant independent of $\rho_0$. Let $\Theta=\tr_{h}h_0$, if we can bound $ \dot\psi$ by a constant  independent of $\rho_0$ by the proof of (i),  we conclude that there is a constant $C_1, \frac12 S_2> S_3>0$ which are independent of $\rho_0$ such that
\bee
\Theta\le C_1
\eee
on $U_{\rho_0}\times[0,S_3]$. Let $\psi$ be the solution of \eqref{e-MA-1} corresponding to $h(t)$, we want to prove that $|\dot\psi|\le C_2$ for some constant independent of $\rho_0$ and on $U_{\rho_0}\times[0,S_3]$. If this is true, then on $U_{\rho_0}\times[0,S_3]$.
$$
C_3^{-1}h_0\le h(t)\le C_3h_0
$$
for some positive constant $C_3$ independent of $\rho_0$. From this one can proceed as in the proof of Theorem \ref{t-existence-1} to show that (i) is true.

In order to estimate $|\dot \psi|$, if one can obtain an upper bound for $\psi$, one can proceed as in the proof of Lemma \ref{l-conformal-1}. To bound $\psi$ from above,
let $$\Phi=\psi-\frac{t}{S_2'}v_1'-At-\e F,$$
$A$ is to be determined.

 If $\sup_{U_{\rho_0}\times[0,S_3]}\Phi>0$, then there is $x_0\in U_{\rho_0}, t_0>0$ such that $\Phi(x_0,t_0)$ attains the maximum.  At this point, we have $\ii(\ddbar \psi-\frac{t}{S_2'}\ddbar v_1)\le \e\ii\ddbar F$, and at $(x_0,t_0)$,
\be\label{e-BKupper}
\begin{split}
0\le \dps{\frac{\p}{\p t}\Phi }
=     \dps{\log \frac{(\omega_0-t\Ric(\omega_0)+\ii\p\bar\p \psi)^n}{\omega_0^n} }-\frac{1}{S_3}v_1-A \\
\end{split}
\ee
On the other hand,
\bee
\begin{split}
\omega_0-t\Ric(\omega_0)+\ii\p\bar\p \psi\le &
 \omega_0-t\Ric(\omega_0)+\frac{t}{S_2'}\ii\ddbar u+\e\ii\ddbar F \\
\le &(\frac t{S_2'}\beta'-1)\omega_0+\e\ii\ddbar F\\
\le&(\frac t{S_2'}\beta'-1+C_4)\omega_0
\end{split}
\eee
for some constant $C_4$ independent of $\rho_0$. Hence \eqref{e-BKupper} impossible, if $A$ is large enough independent of $\rho_0$. From this, it is easy to see that
$$
\psi\le C_5
$$
on $U_{\rho_0}\times[0,S_3]$ for some constant $C_5$ independent of $\rho_0$. This completes the proof of (i).
\end{proof}

 By Lemma \ref{l-equivalent} and Theorems \ref{t-existence-1}, \ref{t-existence-2}, we have
  \begin{cor}\label{c-CR-exist-1}
Let $(M^n,g_0)$ be a complete noncompact \K manifold such that $g_0$ is uniformly equivalent to another \K metric $g_1$ with bounded curvature. If either $\mathrm{BK}(g_0)\ge -K$ or $\mathrm{BK}(g_0)\le K$, then the \KR flow with initial data $g_0$ has short time existence $g(t)$ which is uniformly equivalent to $g_0$.
\end{cor}
\begin{rem}\begin{enumerate}
             \item [(i)] One might compare Corollary \ref{c-CR-exist-1} with some results in \cite{ChauLiTam2016}.
             \item [(ii)] By \cite{WuZheng,YangZheng2013} one can construct $U(n)$ invariant \K metrics on $\C^n$ with nonnegative but unbounded bisectional curvature. By \cite{ChauLiTam2016}, such a metric is uniformly equivalent to a \K metric with bounded curvature. Using their method, one may also construct similar examples with nonpositive bisectional curvature.
           \end{enumerate}

\end{rem}

Next, we want to relax the condition {\bf (a1)} in the existence of Chern-Ricci flow.
 We have the following:

 \begin{thm}\label{t-existence-3} Let $(M^n,g_0)$ be a complete noncompact Hermitian manifold with \K form $\theta_0$ with torsion $T$ such that $|T|_{g_0}, |\bar\p T|_{g_0}$ are uniformly bounded. Assume the following are true:
  \begin{enumerate}
    \item [(i)] There is an exhaustion function $\rho\ge1$ such that
     $$
     \limsup_{x\to\infty}\lf(\rho^{-2}|\p\rho|^2_{g_0}+\rho^{-1}|\ddbar \rho|_{g_0}\ri)=\a<\infty.
     $$
     \item[(ii)] There is $S>0, \beta>0$ and a smooth bounded function $u$ such that
     $$
     \theta_0-S\Ric(\theta_0)+\ii\ddbar u\ge \beta\theta_0.
     $$
    \item [(iii)] $\mathrm{BK}(g_0)\ge -K$ for some $K>0$.
  \end{enumerate}
   Then the Chern-Ricci flow has a solution $g(t)$ with initial data $g_0$  on $M\times[0,S_1)$, where
  $$
S_1=\dps{\frac1 {S^{-1}+nc\a}}.
$$
Here $c>0$ is an absolute constant. In particular, $S_1=S$ if $\a=0$.  Moreover, $g(t)$   is uniformly equivalent to $g$ on $M\times[0,S']$ for all $0<S'<S_1$.

 \end{thm}
 \begin{proof} We may assume that $\beta<1$. As in the proof of Theorem \ref{t-existence-1}, for $\rho_0>>1$, let $U_{\rho_0}$ be the component of $\{\rho(x)<\rho_0\}$ containing a fixed point $x_0\in M$. Let $\mathfrak{F}$  be as in Lemma \ref{l-exhaustion-1} and let $h_0=e^{2F}g_0$ with $F(x)=\mathfrak{F}(\rho(x)/\rho_0)$. For any $\e>0$, by (i) and the properties of $\mathfrak{F}$, there is an absolute constant $c_1>0$ such that
\be\label{e-general-1}
|\p F|^2_{g_0}\le c_1(\a+\e)e^{2F},\ \  |\ddbar F|_{g_0}\le c_1(\a+\e)e^{2F}.
\ee
provided $\rho_0$ is large enough. Let $\omega_0$ be the \K form of $h_0$ then by (ii)
\bee
\begin{split}
\omega_0-S_1\Ric(\omega_0)=&\omega_0-S_1\Ric(\theta_0)+2nS_1\ii\ddbar F\\
\ge& \omega_0+\frac{S_1}S\lf((\beta-1)\theta_0-\ii\ddbar u\ri)-2nc_1(\a+\e)S_1e^{2F}\theta_0\\
\ge&\lf(1+\frac{S_1}S(\beta-1)-2nc_1(\a+\e)S_1\ri)\omega_0-\frac{S_1 }S\ii\ddbar  u\\
\ge&\frac{S_1}S\beta-\frac{S_1 }S\ii\ddbar  u
\end{split}
\eee
if $S_1$ is such that $1-\frac{S_1}S-2nc_1(\a+\e)S_1=0$. That is,
$$
S_1=\lf(S^{-1}+2nc_1(\a+\e)\ri)^{-1}.
$$
 where  have used the fact that $F\ge0$ and  $\beta<1$. By Lemma \ref{l-conformal-a} and \eqref{e-general-1}, we can conclude that the torsion $T_0$ of $h_0$ satisfies $|T_0|_{h_0}, |\bar\p T_0|_{h_0}$ are uniformly bounded by a constant independent of $\rho_0$. Moreover, $\mathrm{BK}(h_0)$ is uniformly bounded below by a constant independent of $\rho_0$. By Theorem \ref{t-existence-1}, we conclude that the Chern-Ricci flow has a solution $h(t)$ on $U_{\rho_0}\times[0,S_1)$ with initial data $h_0$. Moreover, by Lemmas \ref{l-psi-est}, \ref{l-trace-est-1}, we conclude that for any $0<S'<S_1$, there is a constant $C_1>0$ independent of $\rho_0$ such that
$$
C_1h_0\le h(t)\le C_1^{-1}h_0
$$
on $U_{\rho_0}\times[0,S']$. As in the proof of Theorem \ref{t-existence-1}, we conclude that the theorem is true by letting $\e\to0$.
 \end{proof}
 Similarly, we have
 \begin{thm}\label{t-existence-4} Let $(M^n,g_0)$ be a complete noncompact Hermitian manifold with \K form $\theta_0$ with torsion $T$ such that $|T|_{g_0}, |\bar\p T|_{g_0}$ are uniformly bounded. Assume the following are true:
  \begin{enumerate}
    \item [(i)] There is an exhaustion function $\rho\ge1$ such that
     $$
     \limsup_{x\to\infty}\lf(\rho^{-2}|\p\rho|^2_{g_0}+\rho^{-1}|\ddbar \rho|\ri)=\a<\infty.
     $$
     \item[(ii)] There is $S>0, \beta>0, \beta'>0$ and a smooth bounded functions $u, v$ such that
     $$
     \theta_0-S\Ric(\theta_0)+\ii\ddbar u\ge \beta\theta_0.
     $$
     and
     $$
     \theta_0-S\Ric(\theta_0)+\ii\ddbar v\le \beta'\theta_0.
     $$
    \item [(iii)] $\mathrm{BK}(g_0)\le K$ for some $K>0$.
  \end{enumerate}
   Then the Chern-Ricci flow has a solution $g(t)$ with initial data $g_0$ has short time solution  on $M\times[0,S_1)$ for some $S_1>0$. Moreover, $g(t)$   is uniformly equivalent to $g$ on $M\times[0,S_1)$.

 \end{thm}

\section{existence of \KR flow}\label{s-KRF}

In this section, we will use the previous construction of Chern-Ricci flow and the methods in \cite{SimonTopping2017} by Simon-Topping   (who  use some ideas by  Hochard \cite{Hochard2016})   to prove the following:

\begin{thm}\label{t-KRF-existence-1}
Suppose $(M^n,g_0)$ is a complete noncompact \K manifold with complex dimension n with $BK\geq 0$ such that  $V_0(x,1)\geq v_0>0$ for some $v_0>0$ for all $x\in M$.

 \begin{enumerate}
   \item [(i)] There exist $ S=S(n,v_0)>0, a(n,v_0)>0$ depending only on $n, v_0$ such that the \K Ricci flow has a complete  solution $g(t)$  on $M\times[0,S]$ and satisfies
$$  |Rm|(x,t)\leq \frac{a}{t}$$
on $M\times (0,S].$
   \item [(ii)]   $g(t)$ has nonnegative bisectional curvature.
       \item[(iii)] $V_t(x,1)\ge \frac 12v_0$ on $M\times(0,S]$, where $V_t(x,1)$ is the volume of ball of radius 1 centered at $x$ with respect to $g(t)$. If $g_0$ has maximal volume growth, then $g(t)$ also has maximal volume growth.

 \end{enumerate}

\end{thm}

As an application, we give another proof of the following result by Liu \cite{Liu2017}.
\begin{cor}\label{c-uniformization}
Let $(M^n,g_0)$ be a complete noncompact \K manifold with nonnegative bisectional curvature and with maximal volume growth. Then $M$ is biholomorphic to $\C^n$.
\end{cor}
\begin{proof} By   Theorem \ref{t-KRF-existence-1}, $M^n$ also supports a complete \K metric with nonnegative and {\it bounded } bisectional curvature with maximal volume growth. By the result in \cite{ChauTam2006},  $M$ is biholomorphic to $\C^n$.

\end{proof}

To prove the theorem, we introduce   some notations and terminologies:

 \begin{itemize}

 \item Let $(M,g)$ be a Riemannian manifold without boundary which may not be complete. Let $x\in M$, $r>0$. We say that $B(x,r)\subset M$ if the geodesic ball $B(x,r)$ is a subset of $M$. Namely,   $\exp_x$ is defined on the ball of radius $r$ in the tangent space $T_x(M)$ center at the origin.       We say that the injectivity radius $\iota(x)$ of $x$ satisfies $\iota(x)\ge \iota_0$, if $B(x,\iota_0)\subset M$ and $\exp_x$ is a diffeomorphism from the ball of radius $\iota_0$ onto its image $B(x,\iota_0)$.

 \item  Suppose $(M,g_0)$ is a Riemannian manifold (not necessarily complete) and let $g(t)$ be a smooth family of Riemannian metrics on $M$ with $g(0)=g_0$. Then the  geodesic ball  with center at $x$ and radius $r$ with respect to $g(t)$ will be denoted by $B_t(x,r)$ whose volume will be denoted by $V_t(x,r)$.

     \item In this section, all connections and curvature tensors on a Hermitian manifold are referred to the Riemannian connections and Riemannian curvatures, unless specified otherwise.

 \end{itemize}


  In the following Lemmas \ref{l-extension-1}--\ref{l-balls},  we do not assume the manifold is complete.

\begin{lma}\label{l-extension-1} There exists $1>\a>0$ depending only on $n$ so that the following is true: Let $(N^n,g_0)$ be a \K manifold and  $U\subset N$  is a precompact  open set. Let $\rho>0$ be such that $B_0(x,\rho)\subset\subset N$, $|\Rm|(x)\le \rho^{-2}$ and $  \mathrm{inj}_{g_0}(x)\ge \rho$ for all $x\in U$. Assume $U_\rho$ is nonempty. Then for any component $X$ of $U_\rho$, there is a solution $g(t)$ to the \KR flow  on  $X\times[0,\a \rho^2]$, where for any $\lambda>0$
$$
U_\lambda:=\{x\in U|\ B_0(x,\lambda)\subset\subset U\},
$$
with $g(t)$ satisfies the following:
\begin{enumerate}
  \item [(i)] $g(0)=g_0$ on $X$
  \item [(ii)]
  $$
  \a g_0\le g(t)\le \a^{-1}g_0
  $$
  on $X\times[0,\a\rho^2]$.
\end{enumerate}
\end{lma}
\begin{proof} By rescaling, we may assume $\rho=1$. By the proof of Lemma 6.2 in  \cite{Hochard2016}, there is a smooth function
$\sigma(x)\ge0$ on $U$ such that $\sigma(x)=0$ on $U_1$, $\sigma(x)\ge 1$ on $\p U$, $|\nabla \sigma|+|\nabla^2\sigma|\le c_1$. Here and below, lower case $c_i$ will denote a positive constant  depending only on $n$. Let $W=\{x\in U|\ \sigma(x)<1\}$. Then $X\subset U_1\subset W$. Let $\wt W$ be the component of $W$ containing $X$. Let $h_0=e^{2F}g_0$ be the Hermitian metric on $\wt W$ where $F(x)=\mathfrak{F}(\sigma(x))$ where $\mathfrak{F}$ is the function in Lemma \ref{l-exhaustion-1}. Then by  \cite{Hochard2016}, $h_0$ is complete and has bounded curvature. In fact by Lemma  \ref{l-exhaustion-1} and Lemma \ref{l-conformal-a} as in the proof of Lemma \ref{l-conformal-1}, the Chern curvature, the torsion $|T_0|_{h_0}$ and $|\bar\p T_0|_{h_0}$ of $h_0$ are uniformly bounded by a constant $c_2$. By Theorem \ref{t-existence-2}, the Chern-Ricci flow has a solution $h(t)$ on $\wt W\times[0,\a]$ for some $\a>0$ depending only on $n$ such that
$$
\a^{-1}h_0\le h(t)\le \a h_0
$$
on $\wt W\times[0,a]$. Let $g(t)=h(t)|_{X}$. Since $h_0=g_0$ on $X$ because $F=0$ there, it is easy that (i) and (ii) are true.
\end{proof}

\begin{lma}\label{l-curvestimate}
For any $n,v_0>0$, there exist  $\wt S(n,v_0)>0$, $C_0(n,v_0)>0$ depending only on $n, v_0$  such that the following holds: Suppose $(N^{n},g(t))$ is a \K Ricci flow for $t\in [0,S]$ and $x_0\in N$ such  that $B_t(x_0,r)\subset\subset N$ for each $t\in[0,S]$. Suppose
$$V_0(x_0,r^{2n})\geq v_0r^{2n}\quad\text{and}\;\;\mathrm{BK}(g(t))\geq -  r^{-2}\;\;\text{on}\;\;B_t(x_0,r),\;t\in [0,S]$$
Then for all $t\in (0,S]\cap (0,\wt S r^2]$,
$$
|\Rm(x,t)|\le \frac{C_0}{t}\ \ {\rm and}\ \  |\nabla Rm|(x,t)\leq \frac{C_0}{t^{3/2}} \quad\text{on}\;\;B_t(x_0, \frac r8).$$
Moreover the injectivity radius satisfies
$$\mathrm{inj}_{g(t)}(x)\geq (C_0^{-1}t)^\frac12 $$
for $x\in B_t(x_0, \frac r8)$ and $t\in[0,S]\cap[0,\wt S r^2]$.
\end{lma}
\begin{proof} By parabolic rescaling, we may assume that $r=1$. By \cite[Lemma 4.1]{LeeTam2017}, there exist  $\wt S(n,v_0)>0$, $C_0(n,v_0)>0$ such that if $(M,g(t))$ is as in the lemma, then
$$
|\Rm(x,t)|\le \frac{C_0}t
$$
for $x\in B_t(x_0,\frac r2)$ and for $t\in [0,S]\cap [0, \wt S]$. The fact that
$$
\mathrm{inj}_{g(t)}(x)\ge (C_0^{-1} t)^{\frac12}
$$
for $x\in B_t(x_0,\frac 12)$ for $t\in  \in (0,S]\cap (0,\wt S]$ can be proved as in \cite[Lemma 4.1]{SimonTopping2017} using the result of Cheeger-Gromov-Taylor \cite{CheegerGromovTaylor},   volume comparison, and lower bound volume control lemma \cite[Lemma 2.1]{SimonTopping2016} (see also \cite{Simon2012}).

The estimate of $|\nabla\Rm|$ is a consequence of the local estimates by Shi \cite{Shi1989b}  (see \cite{CaoZhuChen2008}).

\end{proof}

The following local estimates are by Sherman-Weinkowve \cite{ShermanWeinkove2012} and Lott-Zhang \cite{LottZhang2013}. The following is from \cite[Propositon A.1]{LottZhang2013}.
\begin{lma}\label{l-curv1}
For any $A_1,n>0$, there exists $C_1(n,A_1)$ depending only on $n, A_1$ such that the following holds: For any \K manifold $(N^n,g_0)$ (not necessarily complete), suppose $g(t), t\in [0,S]$ is a solution of \K Ricci flow on $B_{0}(x_0,r)$ where $B_0(x_0,r)\subset N$ such that on $B_0(x_0,r)$
$$|Rm_{g(0)}|\leq A_1r^{-2} \quad\text{and}\quad |\nabla_{g_0} \Rm(g_0)|\leq A_1r^{-3}. $$
Assume in addition that on $B_{0}(x_0,r)$,
$$A_1^{-1}g_0\leq g(t)\leq A_1 g_0$$
Then on $B_0(x_0,\frac r8)$, $t\in [0,S]$,
$$|\Rm|(g(t))\leq C_1r^{-2}.$$

\end{lma}

\begin{lma}\label{l-lowerbound-BK-1}
Suppose $(N^n,g(t))$ is a \K Ricci flow on $[0,S]$ with $g(0)=g_0$. Let $x_0\in N$, and $r$ be such that $B_t(x_0,r)\subset\subset N$ for all $t\in [0,S]$. Suppose there exists $a>0$
\begin{enumerate}
\item [(i)] $\mathrm{BK}(g_0)\geq 0$ on $B_0(x_0,r)$; and
\item [(ii)] $\displaystyle |\Rm|(g(t))\leq \frac{a}{t}$ on $B_t(x_0,r)$, $t\in (0,S]$.
\end{enumerate}
There exists  $\hat S>0$ depending only on $n, a$ such that for all $t\in [0,S]\cap [0,\hat Sr^2]$, $x\in B_t(x_0,\frac r8)$,
$$\mathrm{BK}(x,t) \geq- r^{-2}.$$
\end{lma}
\begin{proof} This follows from Theorem 3.2 in \cite{LeeTam2017} by rescaling. Note that in \cite{LeeTam2017}, it is assumed that $g(t)$ is completed. However the proof also works for our case.

\end{proof}

We also need following the shrinking balls lemma in \cite{SimonTopping2016}.

\begin{lma}\label{l-balls}
 There exists a constant $\beta=\beta(m)\geq 1$ depending only on $m$ such that the following is true. Suppose $(N^m,g(t))$ is a Ricci flow for $t\in [0,S]$ and $x_0\in N$   with $B_0(x_0,r)\subset\subset M$ for some $r>0$, and $\Ric(g(t))\leq (n-1)a/t$ on $B_0(x_0,r)$ for each $t\in (0,S]$. Then
$$B_t\left(x_0,r-\beta\sqrt{a t}\right)\subset B_0(x_0,r).$$
\end{lma}

We are ready to prove the theorem.

\begin{proof}[Proof of Theorem \ref{t-KRF-existence-1}]

Suppose (i) is true. Then $g(t)$ must have nonnegative bisectional curvature by \cite{LeeTam2017} (see also \cite{HuangTam2015}). Hence (ii) is true. The first part of (iii) follows from (i) and lower bound control lemma \cite[Corollary 6.2]{Simon2012}. If $g_0$ has maximal volume growth, then one may prove similarly by using  parabolic rescaling as in    \cite{HuangTam2015} for example, to conclude that   that $g(t)$ also has maximal volume growth.

To prove (i), let us define some constants:
\begin{itemize}
  \item Let $\a(n)>0$ be the constant in Lemma \ref{l-extension-1}.
  \item Let $C_0(n,v_0)$, $\wt S(n,v_0)$ be the constants in Lemma \ref{l-curvestimate}.
  \item Let $\mu=\mu(n,v_0)=(1+\a C_0^{-1})^\frac12-1>0$.
  \item Let $a=a(n,v_0)=(1+\mu)^2 C_1C_0$ where $C_1=C_1(n,v_0)$ is the constant in Lemma \ref{l-curv1} with $A_1=\max\{C_0,\a^{-1}\}$.
  \item Let $\hat S=\hat S(n,a)=\hat S(n,v_0)$ be the constant in Lemma \ref{l-lowerbound-BK-1}. We may assume that $\wt S\le\hat S  $ by choosing a smaller $\wt S$.
  \item Let $\e=\max\{  32\wt S^{-\frac12}, 4\beta a^\frac12, 2\beta a^\frac12+4C_0^\frac12\}+1$
   .
\end{itemize}

 Fix $p\in M$, for any $R>>1$ we want to construct \KR flow on $B_0(p,R-1)\times[0,S]$ for some $S>0$ depending only on $n,v_0$ such that
$$
|\Rm(g(t))|\le \frac at
$$
on $B_0(p,R-1)\times(0,S]$. Suppose this can be done for all $R>>1$. Denote the corresponding solution by $g_R(t)$. By \cite{Chen2009,Simon2008},  for any domain $\Omega\subset\subset M$   the curvature of $g_R(t)$ are uniformly bounded on $\Omega\times [0,S]$ for $R$ large enough. Hence by the local estimates for higher order derivatives of the curvature tensors of Ricci flow \cite{CaoZhuChen2008,Shi1989a}, we conclude that passing to a subsequence if necessary, $g_R(t)$ converges to a solution $g(t)$ of the \KR flow on $M\times[0,S]$ so that $|\Rm(g(t))|\le \frac at$ on $M\times(0,S]$. The fact that $g(t)$ is complete follows immediately from Lemma \ref{l-balls}.

   Hence it remains to construct \KR flow on $B_0(p,R-1)\times[0,S]$ as mentioned above. Let
    \be\label{e-A}
    A=\min\left\{\frac12\displaystyle{\lf(\frac{\mu}{\e(1+\mu)}\ri)^2}, (32\e^{-1})^2\right\}.
    \ee
    $A$ depends only on $n, v_0$.

     Choose $\rho>0$ small enough so that $|\Rm(g_0)|\le \rho^{-2}$ on $B_0(R+1)$, $\mathrm{inj}_0(x)\ge \rho$ for $x\in B_0(p,R+1)$ and  $B_0(x,\rho)\subset B_0(p,R+2)$ for $x\in B_0(p,R+1)$. Here  $\mathrm{inj}_0(x)$ is the injectivity radius with respect to $g_0$. Apply Lemma \ref{l-extension-1} to $(M,g_0)$ with $U= B_0(p,R+1)$, there is a solution $g(t)$ of the \KR flow with $g(0)=g_0$ in $B_0(p,R)\times[0,t_1]$. We can choose $t_1$ small enough, so that  $t_1\le A$ and
\be\label{e-krf-1}
|\Rm(x,t)|\le \frac at
\ee
for $x\in B_0(p,R)$ and for $t\in (0,t_1]$.

Define $ t_k$, $r_k$   inductively. Let $r_1=0$ and
\be\label{e-tr-def}
t_{k+1}=t_k(1+\mu)^2=t_1(1+\mu)^{2(k-1)};\ \ {\rm and}\ \ r_{k+1} =r_k+ \e t_k^\frac12=\e\sum_{i=1}^kt_i^\frac12.
\ee
It is easy to see that $t_k\to\infty, r_k\to\infty$ as $k\to\infty$.\vskip .2cm

Consider the statement $P(k)$ defined below.\vskip .2cm

{\sl $P(k)$:  There is a solution of the \KR flow $g(t)$ on $B_0(p, R-r_k)\times[0,t_k]$ with $g(0)=g_0$ such that
$$
|\Rm(g(t))|\le \displaystyle{\frac at}
$$
on $B_0(p, R-r_k)\times[0,t_k]$ with $t_k\le A$, $r_k\le 1$.}

\vskip .2cm

From the above, we see that $P(1)$ is true. Let $k$ be the largest integer such that $P(k)$ is true, $t_k\le A$ and $r_k\le 1$.  Then there are three possibilities:

\underline{Case 1}:   $r_{k+1}\ge1$. Then
\bee
\begin{split}
1\le& \e\sum_{i=1}^k t_i^\frac12\\
=& \e t_k^\frac12\sum_{i=1}^k(1+\mu)^{1-i}
\end{split}
\eee
and
$$
A\ge t_k\ge \lf(\frac{\mu}{ \e(1+\mu)}\ri)^2\ge2A
$$
by the definition of $A$. This is impossible.

\underline{Case 2}: $r_{k+1}< 1$ and $t_{k+1}\ge A$.  In this case,
\bee
\begin{split}
A\le t_{k+1}=t_k(1+\mu)^2
\end{split}
\eee
and so $t_k\ge (1+\mu)^{-2}A=:\sigma_1$. Since $P(k)$ is true, there is a solution of the \KR flow on $B_0(p, R-1)\times[0,\sigma_1]$ with $|\Rm(g(t))|\le a/t$ because $r_k\le 1$. Note that $\sigma_1$ depends only on $n, v_0$.

\underline{Case 3}: $r_{k+1}<1$ and $t_{k+1}<A$. We want to prove that this is also impossible.

 Let $g(t)$ be the solution of the \KR flow in $P(k)$. In this case, if $x\in B_0(p,R-r_{k+1}+\frac{1}{2}\epsilon t_k^{1/2})$, then $B_0(x, \frac{1}{2}\e t_k^\frac12)\subset B_0(p,R-r_k)$. Since $P(k)$ is true, by Lemma \ref{l-balls},  applied to $N=B_0(p,R-r_k)$ and $g(t)$, $r=\frac12\e t_k^\frac12$,    for $t\in [0,t_k]$
\bee
B_t(x,\frac14\e t_k^\frac12)\subset B_t(x, \frac12\e t_k^\frac 12-\beta a^\frac12t^\frac12)\subset B_0(x,    \frac12  \e t_k^\frac12)\subset \subset B_0(R-r_k)
\eee
because $\e\ge 4\beta a^\frac12$. By Lemma \ref{l-lowerbound-BK-1}, with   $r=\frac14 \e t_k^\frac12$, we have
\be
\mathrm{BK}(g(t))\ge -(\frac14\e t_k^\frac12)^{-2}\ge -(\frac1{32}\e t_k^\frac12)^{-2}
\ee
 on $B_t(x, \frac1{32}\e t_k^\frac12)$ for all $t\in [0,t_k]$ because $t_k\le \wt S(\frac1{32}\e)^2t_k\le \hat S(\frac1{32}\e)^2t_k \le \hat S(\frac1{4}\e)^2t_k $ by the choice of $\e$. Since $t_k\le A$ we have $\frac1{32}\e t_k^\frac12\le 1$ by \eqref{e-A},  and so $V_0(x,r)\ge v_0 r^{2n}$ with $r= \frac1{32}\e t_k^\frac12$ by volume comparison and the fact that $V_0(x,1)\ge v_0$. By the choice of $\e$, we have $t_k\le \wt Sr^2$. We can apply  Lemma \ref{l-curvestimate} to conclude that
$$
|\Rm(x,t_k)|\le \frac{C_0}{t_k}\ \ {\rm and}\ \  |\nabla Rm|(x,t_k)\leq \frac{C_0}{t_k^{3/2}},$$
and the injectivity radius satisfies
$$\mathrm{inj}_{g(t)}(x)\geq (C_0^{-1}t_k)^\frac12. $$

This is true for all $x\in B_0(p, R-r_{k+1}+\frac12\epsilon t_k^{\frac12})$. By Lemma \ref{l-extension-1},   applied to $U= B_0(p, R-r_{k+1}+\frac12\epsilon t_k^{\frac12})$, initial metric is $g(t_k)$, and $\rho^2=C_0^{-1}t_k$, there is a solution $h(t)$ of the \KR flow on $[t_k, t_k+\a C_0^{-1}t_k]=[t_k,t_{k+1}]$ with $h(t_k)=g(t_k)$ and
$$
\a g(t_k)\le h(t) \le \a^{-1} g(t_k)
$$
on the component $X$  of $p$ in the open set  $$\left\{x|\   B_{t_k}(x, C_0^{-\frac12}t_k^\frac12)\subset\subset   B_0(p, R-r_{k+1}+\frac12\epsilon t_k^\frac12)\right\}.$$

We claim that for $x\in B_0(p,R-r_{k+1})$, $B_{t_k}(x,C_0^{-\frac12}t_k^\frac12)\subset X$. If the claim is true, by Lemma \ref{l-curv1} for all $t\in [t_k,t_{k+1}]$,
$$|Rm(h(t))|(x)\leq \frac{C_0C_1}{t_k}=\frac{(1+\mu)^2C_0C_1}{t_{k+1}}=\frac a{t_{k+1}}\leq \frac{a}{t}.$$

Then if we extend $g(t)$ to be $h(t)$ for $t\in [t_k,t_{k+1}]$,  we see that $P(k+1)$ is true. This contradicts the maximality of $k$ so that $P(k)$ is true. This implies that \underline{Case 3} cannot happen.

To prove the claim, let $x\in B_0(p,R-r_{k+1})$, $y\in B_{t_k}(x,C_0^{-\frac12}t_k^\frac12)$. Since $P(k)$ is true,   by Lemma \ref{l-balls} we have
\begin{align*}
B_{t_k}(y,C_0^{-\frac12}t_k^\frac12)
&\subset B_{t_k}(x,2C_0^{-\frac12}t_k^\frac12)\\
&\subset B_{0}\big (x, (2C_0^{-\frac12}+\beta a^\frac12)t_k^\frac12\big)\\
&\subset B_0\left(p,R-r_{k+1}+(2C_0^{-\frac12}+\beta a^\frac12)t_k^\frac12 \right)\\
&\subset\subset B_0(p,R-r_{k+1}+\frac12 \epsilon t_k^\frac12)
\end{align*}
because $\e>4C_0^{-\frac12}+2\beta a^\frac12$.

Hence if we let $S=A(1+\mu)^{-2}$, we can construct a solution of the \KR flow on $B_0(p,R-1)\times[0,S]$ with $g(0)=g_0$ and $|\Rm(g(t))|\le a/t$. This completes the proof of the theorem.

\end{proof}

\appendix
\section{Chern connection}\label{chern connection appendix}
In this section, we collect some useful formulas for the Chern connection. Those materials can be found in \cite{TosattiWeinkove2015,ShermanWeinkove2013}. Let $(M,g)$ be a Hermitian manifold. The {\it Chern connection} of $g$ is defined as follows: In local holomorphic coordinates $z^i$,
 for a vector field $X_i\p_i$, where $\p_i:=\frac{\p}{\p z^i}$, $\p_{\bar i}=\frac{\p}{\p \bar z^i}$,
 $$\nabla_iX^k=\p_{i}X^k+\Gamma_{ij}^kX^j;\ \nabla_{\bar i}X^k=\p_{\bar i}X^k.
 $$
 For a $(1,0)$ form $a=a_idz^i$,
 $$
 \nabla_ia_j=\p_i a_j-\Gamma_{ij}^ka_k; \ \nabla_{\bar i}a_j=\p_{\bar i}a_j.
 $$
 Here $\nabla_i:=\nabla_{\p_i}$, etc. $\Gamma$ are the coefficients of $\nabla$,
with
$$\Gamma_{ij}^k=g^{k\bar l}\partial_i g_{j\bar l}.$$
Noted that Chern connection is a connection such that $\nabla g=\nabla J=0$.   The {\it torsion} of $g$ is defined to be
$$T_{ij}^k=\Gamma_{ij}^k-\Gamma_{ji}^k.$$
We remark that $g$ is \K if and only if $T=0$. Define the \textit{Chern curvature tensor} of $g$ to be
$$R_{i\bar jk}\,^l=-\partial_{\bar j}\Gamma_{ik}^l.$$
We raise and lower indices by using metric $g$. The Chern-Ricci curvature is defined by
$$R_{i\bar j}=g^{k\bar l}R_{i\bar j k\bar l}=-\partial_i \partial_{\bar j}\log \det g.$$
Note that if $g$ is not \K, then $R_\ijb$ may not equal to $g^{k\bar l}R_{ k\bar li\bar j}$
\begin{lma}
The commutation formulas for the Chern curvature are given by
\begin{align*}
[\nabla_i,\nabla_{\bar j}]X^l=R_{i\bar j k}\,^l X^k,\quad\quad [\nabla_i,\nabla_{\bar j}]a_k=-R_{i\bar j k}\,^l a_l;\\
[\nabla_i,\nabla_{\bar j}]X^{\bar l}=-R_{i\bar j}\,^{\bar k}\,_l X^{\bar l},\quad\quad [\nabla_i,\nabla_{\bar j}]a_k=R_{i\bar j}\,^{\bar k}\,_{\bar l} a_{\bar k}.
\end{align*}
\end{lma}

When $g$ is not K\"ahler, the Bianchi identities maybe fail. The failure can be measured by the torsion tensor.

\begin{lma}\label{l-Chern-connection-1}
In a holomorphic local coordinates, let $T_{ij\bar k}=g_{p\bar k}\Gamma_{ij}^p$,  we have
\begin{align*}
R_{i\bar jk\bar l}-R_{k\bar ji\bar l}&=-\nabla_{\bar j}T_{ik\bar l},\\
R_{i\bar jk\bar l}-R_{i\bar lk\bar j}&=-\nabla_{i}T_{\bar j\bar lk},\\
R_{i\bar jk\bar l}-R_{k\bar li\bar j}&=-\nabla_{\bar j}T_{ik\bar l}-\nabla_kT_{\bar j\bar li}=-\nabla_iT_{\bar j\bar lk}-\nabla_{\bar l}T_{ik\bar j},\\
\nabla_pR_{i\bar jk\bar l}-\nabla_iR_{p\bar jk\bar l}&=-T_{pi}^rR_{r\bar jk\bar l},\\
\nabla_{\bar q}R_{i\bar jk\bar l}-\nabla_{\bar j}R_{i\bar qk\bar l}&=-T_{\bar q\bar j}^{\bar s}R_{i\bar sk\bar l}.
\end{align*}
\end{lma}

The next lemma gives some relations between the Riemannian connection and the Chern connection.

\begin{lma}\label{l-Chern-connection-2} Let $(M^n,g)$ be a Hermitian metric and $\nabla$ be the Chern connection and let $T$ be the torsion.
\begin{enumerate}
   \item [(i)] Suppose $|T|_g, |\nabla T|_g$ are bounded. Then the Riemannian curvature is bounded if and only if the Chern curvature is bounded.
   \item  [(ii)] Let $f$ be a smooth function. Then
\bee
f_{;i\bar j}=f_{i\bar j}-\frac12 g^{k\bar l}g_{p\bar j}T_{il}^p-\frac12 g^{k\bar l}g_{i\bar p}\ol{T_{jl}^p}.
\eee
in a holomorphic local coordinates where $f_{;\ijb}$ is the Hessian with respect to the Riemannian metric.
 \end{enumerate}

\end{lma}

\section{conformal change}

\begin{lma}\label{l-conformal-a} Let $g$ be a Hermitian metric on a complex manifold. Let $h=e^{2F}g$. Then $h$ is also Hermitian. Let $\nabla,\wt \nabla$ be the Chern connections;  $\Gamma, \wt \Gamma$ be the coefficients of Chern connections; $T,\wt T$ be the torsions; $R, \wt R$ be the curvatures of $g, h$ of respectively. In a local holomorphic coordinate neighborhood:
\begin{enumerate}
  \item [(i)] $
\wt\Gamma_{ij}^k=\Gamma_{ij}^k+2F_i\delta_j^k.
$
  \item [(ii)] $\wt T_{ij}^k=T_{ij}^k+2F_i\delta_j^k-2F_j\delta_i^k$. $\wt \nabla_{\bar l}\wt T_{ij}^k=\nabla_{\bar l}T_{ij}^k+2F_{i\bar l}\delta_j^k-2F_{j\bar l}\delta_i^k$.
  \item [(iii)] $\wt R_{r\bar si}^j=R_{r\bar si}^j-2F_{r\bar s}\delta_i^k$; $ \wt R_{k\bar l i\bar j}= e^{2F}\lf(R_{k\bar l i\bar j}-2g_{i\bar j}F_{k\bar l}\ri).$
      \item[(iv)] $\wt R_{k\bar l}=R_{k\bar l}-2nF_{k\bar l}$.

\end{enumerate}
\end{lma}

\end{document}